\documentclass[]{amsart}
\usepackage{amsthm,amsmath,amsfonts,amssymb,amscd,mathrsfs,graphics}
\usepackage{txfonts}
\usepackage{supertabular}
\usepackage{tikz, graphicx}
 \usepackage{multirow}
\usepackage[pdftex,
            pdfauthor={Francis},
            pdftitle={New Methods in FJRW Theory},
            pdfsubject={FJRW theory, Mirror symmetry},
            pdfkeywords={FJRW, Mirror symmetry}]{hyperref}
            
\usepackage{lscape}
\usepackage{url}

\makeatletter
\def\@biblabel#1{[#1]}
\makeatother

\newcommand{\defital}{\textit}
\newcommand{\ds}{\displaystyle}
\newcommand{\Z}{\mathbb Z}
\newcommand{\ZZ}{\mathbb Z}
\newcommand{\C}{\mathbb C}
\newcommand{\CC}{\mathbb C}
\newcommand{\QQ}{\mathbb Q}
\newcommand{\fjrw}[2]{ \left\lceil #1 \:; #2 \right\rfloor }

\newcommand{\zero}{\mathbf 0}

\newcommand{\one}{\mb{1}}

\newcommand{\mb}{\mathbf}

\newcommand{\sH}{\mathscr{H}}
\newcommand{\sA}{\mathscr{A}}
\newcommand{\sI}{\mathscr{I}}
\newcommand{\sC}{\mathscr{C}}
\newcommand{\uC}{\mathfrak{C}}
\newcommand{\sO}{\mathscr{O}}

\newcommand{\sW}{\mathscr{W}}
\newcommand{\sQ}{\mathscr{Q}}
\newcommand{\bMgn}[2]{\overline{\mathscr{M}}_{#1,#2}}

\newcommand{\sL}{\mathscr{L}}

\newtheorem{theorem}{Theorem}
\newtheorem{thm}{Theorem}[section]
\newtheorem{lem}[thm]{Lemma}

\newtheorem{property}{\propertyautorefname}

\newtheorem{cor}[thm]{Corollary}
\newtheorem{conj}[thm]{Conjecture}

\theoremstyle{definition}
\newtheorem{remark}[thm]{Remark}

\newtheorem{defn}[thm]{Definition}
\newtheorem{ax}[thm]{Axiom}
\newtheorem{example}[thm]{Example}
\theoremstyle{remark}
\newtheorem{notat}{Notation} 
\DeclareMathOperator{\Sing}{Sing}
\DeclareMathOperator{\ch}{ch}
\providecommand*{\propertyautorefname}{Property}

\let\oldmarginpar\marginpar
\renewcommand\marginpar[1]{\oldmarginpar[\raggedleft\footnotesize #1]%
{\raggedright\footnotesize #1}}

\begin{document}

\date{\today}

\title[Computational Techniques in FJRW Theory]{Computational Techniques in FJRW Theory\\ with Applications to   Landau-Ginzburg Mirror Symmetry}
\author{Amanda Francis}
\address{Department of Mathematics, Brigham Young University, Provo, UT 84602, USA}
\email{amanda@mathematics.byu.edu}


\begin{abstract}
The Landau-Ginzburg A-model, given by FJRW theory, defines a cohomological field theory, but in most examples is very difficult to compute, especially when the symmetry group is not maximal.  We give some methods for finding the A-model structure.  In many cases our methods completely determine the previously unknown A-model Frobenius manifold structure.  In the case where these Frobenius manifolds are semisimple, this can be shown to determine the structure of the higher genus potential as well.  We compute the Frobenius manifold structure for 27 of the previously unknown unimodal and bimodal singularities and corresponding groups, including 13 cases using a non-maximal symmetry group.

\end{abstract}

\maketitle

\setcounter{tocdepth}{1}
\tableofcontents

\section{Introduction}

Mirror symmetry is a phenomenon from physics that has inspired a lot of interesting mathematics. In the Landau-Ginzburg setting, we have two constructions, the A- and B-models,  each of which depends on a choice of a polynomial with a group of symmetries.  Both models yield Frobenius manifolds.  

The A-model arising from  FJRW theory produces a full cohomological field theory. From the cohomological field theory we can construct correlators and assemble these into a potential function.  This potential function completely determines the Frobenius manifold and determines much of the structure of the cohomological field theory as well.  Although some of these correlators have been computed in special cases, in many cases their computation is quite difficult, especially in the case that the group of symmetries is not maximal (mirror symmetry predicts that these should correspond to an orbifolded B-model).
We give some computational methods for computing correlators, including a formula for concave genus-zero, four-point correlators, and show how to extend these results to find other correlator values.  In many cases our methods give enough information to compute the A-model  Frobenius manifold.  
We give the FJRW Frobenius manifold structure for 27 pairs of polynomials and groups, 13 of which are constructed using a non-maximal symmetry group.

\begin{conj}[The Landau Ginzburg Mirror Symmetry Conjecture] \label{LG_conj}
There exist A- and B-model structures, each constructed from a polynomial $W$ and an associated group, $G$ (see discussion following Definition \ref{G}), such that the A-model for $W$ and $G$ is isomorphic to the B-model for $W^T$ and $G^T$, where $W^T$ and $G^T$ are  dual to the original polynomial and group.
\end{conj}

Although physics predicted its existence, a mathematical construction of the A-model was not known until 2007 when Fan, Jarvis, and Ruan, following the ideas of Witten, proposed a cohomological field theory $\sA_{W,G}$ to satisfy Conjecture \ref{LG_conj} \cite{FJR, FJR1, FJR-ip-WEVFC}.

A basis for the vector space of the A-model consists of pairs of a monomial and a group element.  When the group element acts nontrivially on each variable of $W$, we call this a \textit{narrow} element, otherwise, it is called \textit{broad}. 
The structure in the A-model is determined by certain structure constants called \defital{genus-$g$, $k$-point correlators}, which come from the cohomology of the moduli space $\bMgn{g}{k}$ of genus-$g$ curves with $k$ marked points.  The Frobenius algebra structure is given by the genus-zero, three-point correlators, the Frobenius manifold structure by the genus-zero, $k$-point correlators for $k \geq 3$, and the higher genus structure by the genus-$g$, $k$-point correlators for all nonnegative integers $g$ and $k$ such that $2g-2 +k > 0$.  These correlators are defined as integrals of certain cohomology classes  over $\bMgn{g}{k}$. 
Finding the values of these correlators is a difficult PDE problem, which has not been solved in general.  
They are difficult to compute, especially when they contain broad elements, so in many cases we still do not know how to compute even the A-model Frobenius algebra structure. 
%
 In most cases we do not know how to compute the Frobenius manifold or higher genus structures. In this paper we make progress along these lines.

In 2010, Krawitz \cite{Kr} proved Conjecture \ref{LG_conj} at the level of Frobenius algebra for almost every invertible polynomial $W$ and $G = G^{max}_W$, the maximal symmetry group.  It is more difficult to determine the structure in the A-model when $G \neq G^{max}_W$, because of the introduction of broad elements. 
In 2011 Johnson, Jarvis, Francis and Suggs \cite{FJJS} proved the conjecture at the Frobenius algebra level for any pair $(W,G)$ of invertible polynomial and admissible symmetry group with the following property:

\begin{property}\label{propstar}
Let $W$ be an invertible, nondegenerate, quasihomogeneous polynomial and let $G$ be an admissible symmetry group for $W$.  We say the pair $(W,G)$ has Property $\ref{propstar}$ if 
\begin{enumerate}
\item $W$ can be decomposed as $W = \sum_{i =1}^M W_i$ where the $W_i$ are themselves invertible polynomials having no variables in common with any other $W_j$.
\item For any element $g$ of $G$, where some monomial $\fjrw{m}{g}$ is an element of $\sH_{W,G}$, and for each $i \in \{1, \ldots, M\}$, $g$ fixes either all of the variables in $W_i$ or none of them.
\end{enumerate}
\end{property}

Even for $W$ and $G$ satisfying Property $\ref{propstar}$, where we know the isomorphism class of the the Frobenius algebra, we usually still cannot compute the entire Frobenius manifold (the genus-zero correlators), nor the higher genus potential for the cohomological field theory.
In fact, computing the full structure of either model is difficult, and has only been done in a few cases.  The easiest examples of singularities are the so-called ``simple'' or $ADE$ singularities.  Fan, Jarvis, and Ruan computed the full A-model structure for these in \cite{FJR}.  The next examples come from representatives of the ``Elliptic'' singularities $P_8$, $X_9$, $J_{10}$, and their transpose singularities.  Shen and Krawitz \cite{shenkrawitz} calculated the entire A-model for certain polynomial representatives of $P_8$, $X_9^T$ and $J_{10}^T$, with maximal symmetry group.  

In 2013, Gu{\'e}r{\'e} \cite{Guere} provided an explicit formula for the cohomology classes $\Lambda$ when $W$ is an invertible \defital{chain} type polynomial, and $G$ is the maximal symmetry group for $W$.  

In 2008, Krawitz, Priddis, Acosta, Bergin, and Rathnakumara \cite{KPABR} worked out the Frobenius algebra structure for quasihomogeneous polynomial representatives of various singularities in Arnol'd's list of unimodal and bimodal singularities \cite{Ar}.  We expand on their results, giving the Frobenius manifold structure for the polynomials on Arnol'd's list, as well as their transpose polynomials.  In many cases there is more than one admissible symmetry group for a polynomial.  

In this paper we consider all possible symmetry groups for each polynomial.  Previously, almost no computations have been done in the case where the symmetry group is not maximal, because of the introduction of broad elements. However, this case is particularly interesting because the corresponding transpose symmetry group, used for B side computations, will not be trivial. Thus mirror symmetry predicts these cases will correspond to an orbifolded B-model.
Much is still unknown about these orbifolded B-models. Knowing the corresponding A-model structures, though, give us a very concrete prediction for what these orbifolded B-models should be.

Much of the work that has been done so far in this area has used a list of properties (``axioms'') of FJRW theory originally proved in \cite{FJR} to calculate the values of correlators in certain cases.  Primarily, these axioms have been used to compute genus-zero three-point correlators, but many of them can be used to find information about higher genus and higher point correlators.  One such axiom is the concavity axiom.  This axiom gives a formula for some of the cohomology classes $\Lambda$ of the cohomological field theory in terms of the top Chern class of a sum of derived pushforward sheaves.  

We use a result of Chiodo \cite{Chi} to compute the Chern characters of the individual sheaves in terms of some cohomology classes in the moduli stack of $W$-curves $\sW_{g,k}$ in genus zero.  Then we use various properties of Chern classes to compute the top Chern class of the sum of the sheaves given the Chern character of each individual sheaf.   Much is known about the cohomology $\bMgn{g}{k}$, and not a lot about $\sW_{g,k}$, so we push down the cohomology classes in $\sW_{g,k}$ to certain tautological classes $\psi_i$, $\kappa_a$ and $\Delta_I$ over $\bMgn{0}{k}$.  In this way, we provide a method for expressing $\Lambda$ as a polynomial in the tautological classes of $\bMgn{g}{k}$.  Correlators can then be computed by  integrating $\Lambda$ over $\bMgn{g}{k}$, which is equivalent to calculating certain intersection numbers.  

Algorithms for computing these numbers are well established, for example in \cite{FaMgnF, ACComb}.  Code in various platforms (for example \cite{FaMgnF} in Maple and \cite{Jo} in Sage) has been written which computes the intersection numbers we need.  I wrote code in Sage which performs each of the steps mentioned above to find  the top Chern class of the sum of the derived pushforward sheaves, and then uses Johnson's intersection code \cite{Jo} to find intersection numbers. This allows us to compute certain correlator values which were previously unknown.  In particular, in Lemma \ref{4pt_concave}, we restate an explicit formula  found in \cite{FJR} for computing any concave genus-zero four-point correlators, with a proof not previously given.  We also describe how to compute higher point correlators. To do this, we use a strengthened version of the Reconstruction Lemma of \cite{FJR} to find values of non-concave correlators, with an aim to describe the full Frobenius manifold structure of many  pairs  $(W,G)$ of singularities and groups. In many cases, these new methods allow us to compute Frobenius manifold structures for certain singularities and groups which were previously unknown.

In 2014, Li, Li, Saito, and Shen \cite{LLSS} computed the entire A-model structure for the 14 polynomial representatives of exceptional unimodal singularities listed in \cite{Ar}, and their transpose polynomial representatives, with \defital{maximal symmetry group}. In fact, for these polynomials and groups, they proved Conjecture \ref{LG_conj}.  But there are also three non-maximal symmetry groups that appear in this family of polynomials:  the minimal symmetry groups for $Z_{13}^T$, $Q_{12}$, and $U_{12}$.  The A-model for the first of these is computed in this paper (see Section \ref{sec_done}), for the second, the A-model splits into the tensor product of two known A-models (see Section \ref{sec_splits}).  We do not currently know a way to compute the A-model structure for $U_{12}$ with minimal symmetry group (see Section \ref{sec_nogo}).

There are 51 pairs of invertible polynomials and admissible symmetry groups corresponding to those listed in  \cite{Ar}, whose A-model structure is still unknown.  Of these, 16 pairs have FJRW theories which split into the tensor products of theories previously computed (see Section \ref{sec_splits}).  For 27 of the remaining pairs, we are able to use our computational methods to find the full Frobenius manifold structure of the corresponding A-model (see Section \ref{sec_done}), including 13 examples with non-maximal symmetry groups.  There are 8 pairs whose theories we still cannot compute using any known methods (see Section \ref{sec_nogo}).

\subsection{Acknowledgements}
I would like to thank my advisor Tyler Jarvis for many helpful discussions that made this work possible. Thank also to those who have worked on the SAGE code which I used to do these computations, including Drew Johnson, Scott Mancuso and Rachel Webb.  I would also like to thank Rachel Webb for pointing out some errors in earlier versions of this paper. 
\section{Background}\label{review}
We begin by reviewing key facts about the construction of the A-models. This will require the choice of an admissible polynomial and an associated symmetry group.

\subsection{Admissible Polynomials and Symmetry Groups}
A polynomial $W \in \mathbb{C}[x_1, \ldots, x_n]$, where $W = \sum_i c_i\prod_{j = 1}^n x_j^{a_{i,j}}$ 
is called \defital{quasihomogeneous} if there exist positive rational numbers $q_j$ (called weights) for each variable $x_j$ such that each monomial of $W$ has weighted degree one.  That is, for every $i$ where $c_i \neq 0$, 
$
\sum_j q_j a_{i,j} = 1.
$

A polynomial $W \in \mathbb{C}[x_1, \ldots, x_n]$ is called \defital{nondegenerate} if it has an isolated singularity at the origin.

Any polynomial which is both quasihomogeneous and nondegenerate will be considered \defital{admissible} for our purposes.  
We say that an admissible polynomial is \defital{invertible} if it has the same number of variables and monomials. 
In this paper, we focus on computations involving only {invertible} polynomials.  

The \defital{central charge} $\hat{c}$ of an admissible polynomial $W$ is given by
$
\hat{c} = \sum_j (1 - 2q_j).
$

Next we define the maximal diagonal symmetry group.

\begin{defn}\label{G}
Let $W$ be an admissible polynomial. The \defital{maximal diagonal symmetry group} $G_W^{max}$ is the group of elements of the form $g = (g_1, \ldots, g_n) \in \left(\mathbb{Q}/\ZZ\right)^n$ such that
\[
W(e^{2\pi i g_1}x_1,e^{2\pi i g_2}x_2,\ldots,e^{2\pi i g_n}x_n) = W(x_1,x_2,\ldots,x_n).
\] 
\end{defn}

\begin{remark}\label{J_elem}
If $q_1, q_2, \ldots, q_n$ are the weights of $W$, then the element $J = (q_1, \ldots, q_n)$ is an element of $G_W^{max}$.  
\end{remark}

FJRW theory requires not only a quasihomogeneous, nondegenerate polynomial $W$, but also the choice of a subgroup $G$ of $G_W^{max}$ which contains the element $J$.  Such a group is called \defital{admissible}.  We denote $\left\langle J\right> = G^{min}_W$, and any subgroup between $G^{min}$ and $G^{max}$ is admissible.

\begin{example}\label{Example_WG}
Consider the polynomial $W=W_{1,0} = x^4 + y^6$.  In this case that there are two admissible symmetry groups:
\[
\left\langle J\right\rangle= \langle(1/4,1/6 )\rangle, 
\text{ and } 
G_W^{max} = \langle (1/4,0), (0,1/6) \rangle.
 \]
 We shall use $W = W_{1,0}$ and $G = \langle J \rangle$
 for the rest of the examples in this section.

\end{example}

\subsection{Vector Space Construction}
We now briefly review the construction of the A-model state space, as a graded vector space.

We use the notation $I_g = \{i \hspace{.1cm} |\hspace{.1cm}  g \cdot x_i = x_i\}$ to denote the set of indices of those variables fixed by an element $g$, and  $Fix(g)$ to denote the subspace of $\C^n$ which is fixed by $g$,
$
Fix(g) =\{ (a_1, \ldots, a_n) | \text{ such that } a_i = 0 \text{ whenever } g \cdot x_i \neq x_i \}
$.
The notation $W_g$ will denote the polynomial $W$ restricted to $Fix(g)$.  

For the polynomial $W$ with symmetry group $G$, the state space $\sH_{W,G}$ is defined in terms of Lefschetz thimbles
and is equipped with a natural pairing, $\langle \bullet, \bullet \rangle : \sH_{W,G} \times \sH_{W,G} \rightarrow \C$ which we will not define here.
\begin{defn}
Let $\sH_{g,G}$ be the $G$-invariants of the middle-dimensional relative cohomology
\[
\sH_{g,G} = H^{mid} (Fix(g),(W)_g^{-1}( \infty))^G,
\]
where $W_g^{-1}(\infty)$ is a generic smooth fiber of the restriction of $W$ to $Fix(g)$.   The state space is given by
\[
\sH_{W,G} = \left(\bigoplus_{g \in G} \sH_{g,G}\right).
\]
\end{defn}

Recall the \defital{Milnor Ring} $\sQ_W$ of  the polynomial $W$ is $\ds{\frac{\C[x_1 \ldots, x_n]}{
\left\langle\frac{\partial W}{\partial x_1}, \ldots, \frac{\partial W}{\partial x_n}\right\rangle
}
}$. It has a natural residue pairing.

\begin{lem}[Wall]\label{Wall}
Let $\omega = dx_1 \wedge \ldots \wedge dx_n$, then 
\[
\sH_{\zero,G}  = H^{mid} (\CC^n,(W)^{-1}( \infty)) \cong  \frac{ \Omega^n}{dW \wedge\Omega^{n-1}} \cong \sQ_W \omega.
\]
are isomorphic as $G_W$-spaces, and this isomorphism respects the pairing on both. 
\end{lem}

The isomorphism in Lemma \ref{Wall} certainly will hold for the restricted polynomials $W_g$ as well.  This gives us the useful fact
\begin{equation}\label{state_space}
\sH_{W,G} = \bigoplus_{g \in G} \sH_{g,G}  \cong \bigoplus_{g \in G} \left(\sQ_g \omega_g \right)^G,
\end{equation}
where $\omega_g = dx_{i_1} \wedge \ldots \wedge dx_{i_s}$ for $i_j \in I_g$.

\begin{notat}\label{FJRWNotation}
An element of $\mathscr{H}_{W,G}$ is a linear combination of basis elements. We denote these basis elements by $\fjrw{m}{g}$, where $m$ is a monomial in $\C[x_{i_1}, \ldots, x_{i_r}]$ and $\{i_1, \ldots, i_r\} = I_g$.  We say that $\fjrw{m}{g}$ is \defital{narrow} if $ I_g = \emptyset$ , and \defital{broad} otherwise.
\end{notat}

The \defital{complex degree} $\deg_{\C}$ of a basis element $\alpha = \fjrw{m}{(g^1, \ldots, g^n)}$ is given by 
\[
\deg_{\C} \alpha = \frac{1}{2}N_i + \sum_{j = 1}^n (g^j - q_j), 
\]
where $N_g$ is the number of variables in $Fix(g_i)$.

By fixing an order for the basis, we can create a matrix which contains all the pairing information.
This \defital{pairing matrix} $\eta$ is given by 
\[
\eta = \left[\left\langle \alpha_i,\alpha_j\right\rangle\right]
\]
where $\{\alpha_i\}$ is a basis for $\sH_{W,G}$.

\subsection{The Moduli Space and basic properties}
The FJRW cohomological field theory arises from the construction of certain cohomology classes on a finite cover of the moduli space of stable curves, the moduli space of $W$-orbicurves. 

\subsubsection{Moduli Spaces of Curves}

The \defital{Moduli Space of stable curves of genus $g$ over $\CC$ with $k$ marked points $\bMgn{g}{k}$} can be thought of as the set of equivalence classes of   
(possibly nodal) Riemann surfaces $C$ with $k$ marked points, $p_1, \ldots p_k$, where $p_i \neq p_j$ if $i \neq j$.  We require an additional stability condition,  that the automorphism group of any such curve be finite.  This means that $2g-2+k >0$ for each irreducible component of $C$, where $k$ includes the nodal points.  

We denote the universal curve over $\bMgn{g}{k}$ by $\uC \xrightarrow{\pi} \bMgn{g}{k}$.

The \defital{dual graph} of a curve in $\bMgn{g}{k}$ is a graph with a node representing  each irreducible component, an edge for each nodal point, and a half edge for each mark. Componenents with genus equal to zero are denoted with a filled-in dot while higher-genus components are denoted by a vertex labeled  with the genus.

\begin{example}
A nodal curve in $\bMgn{1}{3}$ and its dual graph are shown below.
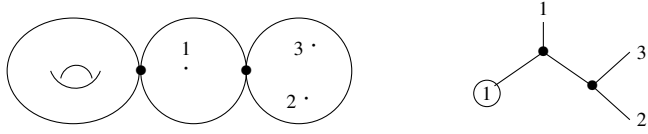
\begin{center}
\begin{tikzpicture}

\draw (-5,0) ellipse (25pt and 20pt);
\node at (-4.1,0) {$\bullet$};
\draw (-3.4,0) ellipse (20pt and 20pt);
\node at (-2.7,0) {$\bullet$};
\draw (-2,0) ellipse (20pt and 20pt);
\draw (-5.3,0) arc (-160: -20: 10pt);
\draw (-4.75,-.1) arc (10: 170: 6pt);
\node[scale = .75] at (-2,.3) {$3$};
\node at (-1.9,-.4) {$\cdot$};
\node[scale = .75] at (-2.1,-.4) {$2$};
\node at (-1.8,.3) {$\cdot$};
\node[scale = .75] at (-3.5,.3) {$1$};
\node at (-3.5,0) {$\cdot$};

\node[scale = .75] at (.5,-.3) {$1$};
\draw (.5,-.3) circle (5 pt);
\draw (.6,-.2) --(1.25,.25);
\node at (1.25,0.25) {$\bullet$};
\draw (1.25,.65)node[anchor = south, scale = .75]{$1 $}--(1.25,.25);
\draw (1.9,-.2) --(1.25,.25);
\node at (1.9,-.2) {$\bullet$};
\draw (1.9,-.2) -- (2.4,.25) node[anchor = west, scale = .75] {$3$};
\draw (2.4,-.65) node[anchor = west, scale = .75] {$2$}--(1.9,-.2);
\end{tikzpicture}
\end{center}

\end{example}

For each pair of non-negative integers $g$ and $k$, with $2g - 2 + k > 0$, the FJRW cohomological field theory produces for each $k$-tuple $(\alpha_1, \ldots, \alpha_k) \in \sH_{W,G}^{\otimes k}$ a cohomology class $\Lambda^W_{g,k}(\alpha_1,\alpha_2,...,\alpha_k) \in H^*(\bMgn{g}{k})$.  The definition of this class can be found in \cite{FJR1}.

A genus-$g$, $k$-point \defital{correlator} with insertions $\alpha_1, \ldots, \alpha_k \in \sH_{W,G}$ is defined by the integral
\[
\left\langle \alpha_1, \ldots, \alpha_k\right>_{g,n} = \int_{\bMgn{g}{k}} \Lambda_{g,k}\left(\alpha_1, \ldots, \alpha_k\right).
\]

Finding the values of these correlators is a difficult PDE problem, which has not been solved in general. In spite of this, we can calculate enough of them to determine the full Frobenius manifold structure in many cases.
Fan, Jarvis and Ruan \cite{FJR} provide some axioms the $\Lambda$ classes must satisfy, which, in some cases,  allow us to determine their values. For example, these axioms give the following selection rules for non-vanishing correlators.

\begin{ax}[FJR]
\label{sel_rules}

Here we let  $\alpha_i = \fjrw{m_i}{g_i}$, with $g_i = (g_i^1, \ldots, g_i^n)$. 
\begin{enumerate}
\item
$\left\langle \alpha_1, \ldots, \alpha_k\right>_{g,k} = \left\langle \alpha_{\sigma(1)}, \ldots, \alpha_{\sigma(k)}\right>_{g,k}$. For any permutation $\sigma \in S_k$.
\item  
$ \left\langle \alpha_1, \ldots, \alpha_k\right>_{0,k} = 0$ unless
 $\sum_{i = 1}^k\deg_{\C}\alpha_i = \hat{c} + k - 3$.
\item $\left\langle \alpha_1, \ldots, \alpha_k\right>_{0,k} = 0$ unless $q_j(k-2) - \sum_{i = 1}^k g_i^j \in \Z$ for each $j = 1, \ldots n$
\end{enumerate}
\end{ax}

Notice that the last selection rule above gives the following lemma.

\begin{lem}\label{lem_lastelement}
The correlator $ \left\langle \alpha_1, \ldots, \alpha_k\right>_{0,k} = 0$ unless
 $g_k = (k-2)\cdot J - \sum_{i=1}^{k-1} g_i$, 
 where $J \in G^{max}_W$ is the element whose entries are the quasihomogeneous weights of $W$.
 \end{lem}
The following splitting axiom will also be useful for us.
\begin{ax}[FJR]
\label{split_axiom}
 if $W_1 \in \CC[x_1, \ldots, x_r]$ and $W_2 \in \CC[x_{r+1}, \ldots, x_n]$ are two admissible polynomials with symmetry groups $G_1$, and $G_2$, respectively, then 

\[
\sH_{W_1+W_2, G_1 \oplus G_2} \cong \sH_{W_1, G_1} \otimes \sH_{W_2, G_2}
\]

where the $\Lambda$ classes are related by 
\[
\Lambda^{W_1 + W_2}(\alpha_1 \otimes \beta_1, \ldots, \alpha_k \otimes \beta_k) = \Lambda^{W_1 }(\alpha_1 , \ldots, \alpha_k ) \otimes 
\Lambda^{ W_2}(\beta_1, \ldots, \beta_k).
\]
\end{ax}
We omit the rest of the axioms (except the Concavity Axiom, which we will discuss later), but refer the reader to \cite{FJR} for the axioms, and \cite{KPABR} for a detailed explanation of how to use them to find genus-zero, three-point correlator values. 

\subsubsection{B-model Structure}

For the unorbifolded B-model the Frobenius manifold is given by the Saito Frobenius manifold for a particular choice of primitive form (see \cite{Saito}).  For the orbifolded B-model the Frobenius manifold structure is still unknown.

\section{Computational Methods}\label{concavity}
Here we give results for computing concave genus-zero correlators and discuss how to use the reconstruction lemma to find values of other correlators.
\subsection{Using the Concavity Axiom}\label{concavity_section}
We give a formula for $\Lambda$ as a polynomial in the tautological classes $\psi_i$, $\kappa_a$, and $\Delta_I$ in $H^*(\bMgn{g}{k})$. 
We then give a  formula for computing concave genus-zero four-point correlators.

\begin{ax}[FJR]
Suppose that all $\alpha_i$ are narrow insertions of the form $\fjrw{1}{g}$  (See Notation \ref{FJRWNotation}).  If $ \pi_* \bigoplus_{i=1}^n \sL_i = 0$, then the cohomology class $\Lambda_{g,k}^W(\alpha_1, \ldots, \alpha_k) $ can be given in terms of the top Chern class of the derived pushforward sheaf $R^1 \pi_* \bigoplus_{i=1}^n \sL_i$:
\begin{equation}\label{concave_lambda}
\Lambda_{g,k}^W(\alpha_1, \ldots, \alpha_k) = \frac{|G|^g}{\deg (st)} PD st_* \left( 
PD^{-1}\left((-1)^Dc_D\left(R^1 \pi_* \bigoplus_{i=1}^n \sL_i\right)\right)
\right)
\end{equation}
\end{ax}

Recall that integration of top dimensional cohomology classes is the same as pushing them forward to a point, or computing intersection numbers.

In this section we express $(-1)^D c_D R^1\pi_{*}(\sL_1 \oplus \ldots \oplus \sL_N)$ as a polynomial $f$ in terms of pullbacks of $\psi$,  $\kappa$, and $\Delta_I$ classes,
then, which gives 
\[\begin{array}{lr}
\langle\alpha_1, \ldots, \alpha_k \rangle& = p_* \frac{|G|^g}{\deg (st)} PD st_* \left( 
PD^{-1}\left(st^* f(\kappa_1, \ldots, \kappa_D, \psi_1, \ldots, \psi_k, \{\Delta_I\}_{I \in \sI})
\right)\right)\\
	& = |G|^g p_* \left( f(\kappa_1, \ldots, \kappa_D, \psi_1, \ldots, \psi_k, \{\Delta_I\}_{I \in \sI})\right),
	\end{array}
\]
and this allows us to use intersection theory to solve for these numbers.

Recall the following well-known property of chern classes from K-theory.
\begin{equation}\label{chernclassprop}
\ds{c_t(\mathscr{E}) = \frac{1}{c_t(-\mathscr{E})} = \frac{1}{1-(-c_t(-\mathscr{E}))} = \sum_{i=0}^\infty (-c_t(-\mathscr{E}))^i}
\end{equation}

A vector bundle $\mathbb{E}$ is concave when  $R^0 \pi_*\mathbb{E} = 0$, or when
\[
R^\bullet \pi_* \mathbb{E} = R^0 \pi_*\mathbb{E}-R^1 \pi_*\mathbb{E} = -R^1 \pi_*\mathbb{E}
\]
Equation \ref{chernclassprop} gives 
\[
c_t(R^1 \pi_*\bigoplus_i \sL_i) = \sum_{j=0}^{\infty}\left(-c_t(R^\bullet \pi_*\bigoplus_i \sL_i)\right)^j 
\]

Also, since the total Chern classes are multiplicative, we have,
\begin{equation}\label{chernclasssplit}
\ds{c_k(R^\bullet \pi_*\bigoplus_i \sL_i) = \sum_{\sum i_j = k}\left( \prod_{j = 1}^n c_{i_j}(R^\bullet \pi_*\sL_{j})\right)}.
\end{equation}
Together Equations \ref{chernclassprop} and \ref{chernclasssplit} give a formula for finding the $i$th chern class of $R^1 \pi_*\bigoplus_i \sL_i$ in terms of the chern classes of the $R^\bullet \pi_*\sL_{j}$.

It is well-known that Chern characters can be expressed in terms of Chern classes, but
it is also possible to express Chern classes in terms of Chern characters (for example, in  \cite{EJK}).  We have
\begin{equation}\label{chartoclass}
\begin{array}{lr}
c_t(R^\bullet \pi_*\sL_{k}) &=  \exp\left( \sum_{i = 1}^\infty (i-1)! (-1)^{i-1}ch_i(R^\bullet \pi_*\sL_{k}) t^i \right)\\
			& \sum_{j = 0}^\infty \frac{1}{j!}\left( \sum_{i = 1}^\infty (i-1)! (-1)^{i-1}ch_i(R^\bullet \pi_*\sL_{k}) t^i \right)^j 
\end{array}
\end{equation}

 Thus, there exists a polynomial $f$ such that 
\begin{equation}\label{chern_to_char}
c_D(R^1 \pi_*\bigoplus_i \sL_i) = f(ch_1(R^\bullet \pi_*\sL_{k}), \ldots, ch_D(R^\bullet \pi_*\sL_{k}))
\end{equation}

In Section \ref{chiodo} we will use a result of Chiodo \cite{Chi} to express $ch_i(R^{\bullet} \pi_* \sL_k)$ in terms of tautological cohomology classes on $\bMgn{g}{n}$ and this will allow us to compute the corresponding correlators, but first we need to discuss some cohomology classes on $\sW_{g,n}$ and $\bMgn{g}{n}$ and their relations.  

\subsubsection{Orbicurves}

An \defital{orbicurve} $\sC$ with marked points $p_1 , \ldots , p_k$ is a stable curve $C$ with orbifold structure at each $p_i$ and each node. Near each marked point $p_i$ there is a local group action given by $\ZZ/m_i$ for some positive integer $m_i$. 
Similarly, near each node $p$ there is again a local group $ \ZZ/n_j$ whose action on one branch is inverse to the action on the other branch. 

In a neighborhood of $p_i$, $\sC$ maps to $C$ via the map,  
\[
\varrho : \sC \rightarrow C,
\]
where if $z$ is the local coordinate on $\sC$ near $p_i$, and $x$ is the local coordinate on $C$ near $p_i$, then $\varrho(z) = z^{r} = x$.

Let  $K_C$ be the canonical bundle of $C$. The \defital{log-canonical bundle} of $C$ is the line bundle
\[
K_{C,log} =K_C\otimes \sO(p_1)\otimes \ldots \otimes \sO(p_k),
\]
where $\sO ( p_i)$ is the holomorphic line bundle of degree one whose sections may have a simple pole at $p_i$.

The log-canonical bundle of $\sC$ is defined to be the pullback to $\sC$ of the log-canonical bundle of $C$:
\[
K_{\sC,log} =\varrho^* K_{C,log}.
\]

Given an admissible polynomial $W$, a $W$-structure on an orbicurve $\sC$ is essentially a
choice of $n$ line bundles $\sL_1,...,\sL_n$ so that for each monomial of $W = \sum_j M_j$, with $M_j =x_1^{a{j,1}} \ldots x_n^{a_{j,n}}$, 
we have an isomorphism of line bundles
\[
\varphi_j:  \sL_1^{\otimes a_{j,1}} \ldots \sL_n^{\otimes a_{j,n}} \tilde{\longrightarrow} K_{\sC, log}.
\]

Recall that Fan, Jarvis and Ruan \cite{FJR} defined a stack 
\[
\mathscr{W}_{W,g,k} = \{\mathscr{C}, p_1, \ldots, p_k, \sL_1, \ldots, \sL_N, \varphi_1 \ldots \varphi_s\}
\]
 of stable orbicurves with the additional $W$-structure, and the canonical morphism, 
 \begin{center}
\begin{tikzpicture}
\node at (0,0) {$\mathscr{W}_{g,k}$};
\put(15,0) {\vector(1,0){40}};
\node at (2.5,0) {$\mathscr{M}_{g,k}$};
\node at (1.25,.3) {$st$};
\end{tikzpicture}
 \end{center}	
from the stack of $W$-curves to the stack of stable curves, $\mathscr{M}_{g,k}$. 


Notice that  $\varrho_*$ will take global sections to global sections, and a straightforward computation (see \cite{FJR} \S2.1 ) shows that  if $\sL$ on $\sC$ such that $\sL^{\otimes r } \cong K^{\otimes s}_{\sC,\log}$, then 
\[
(\varrho_* \sL)^r = K_{C,\log}^{\otimes s}(-(r-m_i)p).
\]

Recall that if $\sL$ is the line bundle associated to the sheaf $\mathcal{L}$ and the action at an orbifold point $p_i$ 
on $\sL$ is given in local coordinates by $(z,v) \mapsto (\zeta z, \zeta^{m_i}v)$, then the action on $\mathcal{L}$ will take a generator $s$ of $\mathcal{L}$ near $p_i$ and map it to $\zeta^{(r - m_i)}s$. 
Thus, if $\sL^r \cong K_{\sC, log}^{\otimes s}$ on a smooth orbicurve with action of the local group on $\sL$ defined by $\zeta^{m_i}$ for $m_i >0$ at each marked point $p_i$, then 
\[
\left( \varrho_* \mathcal{L} \right)^r = |\mathcal{L} |^r = \omega_{C,log}^{\otimes s} \otimes \left( \bigotimes_i \sO((-m_i)p_i)\right).
\]

\begin{example}
Suppose $W$ and $G$ are as in Example \ref{Example_WG} with notation as in Notation \ref{FJRWNotation}. If, for a curve in $\bMgn{0}{4}$ the four marked points correspond to the A-model elements 
$\fjrw{1}{(1/4,1/2)}$, $\fjrw{1}{(1/4,1/2)}$, $\fjrw{1}{(1/4,1/2)}$, and $\fjrw{1}{(3/4,5/6)}$, then 
\[\begin{array}{c}
|\sL_x|^{4} = \omega_{C,log} \otimes  \sO((-1)p_1)\otimes  \sO((-1)p_2)\otimes \sO((-1)p_3) \otimes \sO((-3)p_4)\\
|\sL_y|^6 = \omega_{C,log} \otimes  \sO((-3)p_1)\otimes  \sO((-3)p_2)\otimes \sO((-3)p_3) \otimes \sO((-5)p_4)\\
\end{array}\]
\end{example}

\subsubsection{Some Special Cohomology Classes in $\bMgn{g}{k}$ and $\sW_{g,k}$}
For $i \in \{1, \ldots, k\}$, $\psi_i \in H^q(\bMgn{g}{k})$ is the first Chern class of the line bundle whose fiber at $(C, p_1, \ldots, p_k)$ is the cotangent space to $C$ at $p_i$. 
In other words, if $\pi_{k+1}:\bMgn{g}{k+1} = \uC \rightarrow \bMgn{g}{k}$ is the universal curve, and it is also the morphism obtained by forgetting the $(k+1)$-st marked point, $\omega_{\pi_{k+1}}=\omega$ is the relative dualizing sheaf, and $\sigma_i$ is the section of $\pi_{k+1}$ which attaches a genus-zero, three-pointed curve to $C$ at the point $p_i$, and then labels the two remaining marked points on the genus-zero curve $i$ and $k+1$,
 \begin{center}
\begin{tikzpicture}
\node at (-.1,0) {$\bMgn{g}{k+1}$};
\put(0,-12){\vector(0, -5){30}};
\node at (-.1,-1.9) {$\bMgn{g}{k}$};
\node at (.4,-.75) {$\pi_{k+1}$};
\draw[->] (-.5,-1.5) arc (-160: -200: 50pt);
\node at (-1, -.75) {$\sigma_i$};

\end{tikzpicture}
 \end{center}	
then, $\mathbb{L} = \sigma^*(\omega)$ is the \textit{cotangent line bundle} and its first Chern class is $\psi_i$:
\[
\psi_i = c_1(\sigma^*(\omega)), 
\]

Let $D_{i,k+1}$ be the image of $\sigma_i$ in $\bMgn{g}{k+1}$, then we define 
\[
K = c_1\left(\omega\left(\sum_{i = 1}^k D_{i,k+1}\right)\right) ,
\]
and for $ a \in \{1, \ldots, 3g-3+k\}$, 
\[
\kappa_a = \pi_*(K^{a+1}).
\]

Each partition $I \sqcup J = \{1, \ldots, k\}$ and $g_1 + g_2 = g$ of marks and genus  such that $1 \in I$, $2g_1 - 2 + |I|+1>0$ and $2g_2 - 2 + |J|+1>0$, gives an irreducible boundary divisor, which we label $\Delta_{g_1, I}$.
These boundary divisors are the nodal curves in $\bMgn{g}{k}$.  For example, the boundary divisor $\Delta_{1,\{1,2\}}$ in $\bMgn{1}{5}$ and its dual graph are given below.
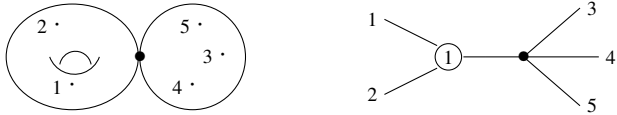
\begin{center}
\begin{tikzpicture}

\draw (-5,0) ellipse (25pt and 20pt);
\node at (-4.1,0) {$\bullet$};
\draw (-3.4,0) ellipse (20pt and 20pt);
\draw (-5.3,0) arc (-160: -20: 10pt);
\draw (-4.75,-.1) arc (10: 170: 6pt);
\node at (-5,-.4) {$\cdot$};
\node[scale = .75] at (-5.2,-.4) {$1$};
\node at (-5.2,.4) {$\cdot$};
\node[scale = .75] at (-5.4,.4) {$2$};
\node at (-3,0) {$\cdot$};
\node[scale = .75] at (-3.2,0) {$3$};
\node at (-3.4,-.4) {$\cdot$};
\node[scale = .75] at (-3.6,-.4) {$4$};
\node at (-3.3,.4) {$\cdot$};
\node[scale = .75] at (-3.5,.4) {$5$};

\draw (-.85,.5) node[anchor = east,scale = .75] {$1$}--(-.15,.15);
\draw (.2,0) --(.5,0);
\node[scale = .75] at (0,0) {$1$};
\draw (0,0) circle (5 pt);
\draw (.5,0) --(1,0);
\draw (1,0) -- (1.75,.65) node[anchor = west, scale = .75] {$3$};
\node at (1,0) {$\bullet$};
\draw (-.85,-.5)node[scale = .75, anchor = east]{$2 $}--(-.15,-.15);
\draw (2,0) node[anchor = west, scale = .75] {$4$}--(1,0);
\draw (1.75,-.65)node[anchor = west, scale = .75]{$5 $}--(1,0);
\end{tikzpicture}
\end{center}

We will use the following well-known lemma for $\bMgn{0}{k}$ for expressing $\psi$ classes in terms of boundary divisors.
\begin{lem}\label{psi_relation}
\[
\psi_i = \sum_{\substack{a \in I\\ b,c \notin I}} \Delta_I
\]
\end{lem}

Now we consider some cohomology classes on $\sW_{g,k}$. Consider the diagram:
\begin{center}
\begin{tikzpicture}
\node at (0,0) {$\uC_{g,k}$};
\draw[->] (.1,-.4) -- (.1, -1.6);
\node at (.35, -1) {$\pi$};
\draw[<-] (-.1,-.4) -- (-.1, -1.6);
\node at (-.45, -1) {$\sigma_i$};
\node at (0,-2) {$\sW_{g,k}$};
\draw[->] (.5,-2) -- (1.5, -2);
\node at (1, -1.8) {$st$};
\node at (2,-2) {$\bMgn{g}{k}$};
\end{tikzpicture}
\end{center}

where $\uC_{g,k}$ is the universal orbicurve over $\sW_{g,k}$.

The stack $\sW_{g,k}$ has cohomology classes $\tilde{\psi}_i$, $\tilde{\kappa}_a$ and $\tilde{\Delta}_I$ defined in the same manner as $\psi_i$, $\kappa_a$ and $\Delta_I$ in $\bMgn{g}{k}$. They satisfy the following properties (see \S 2.3 in \cite{FJR}).

\begin{equation}\label{sMvssWprop}
\tilde{\psi}_i  = st^*(\psi_i), \quad 
\tilde{\kappa}_a  = st^*(\kappa_a), \quad 
r \tilde{\Delta_I} = st^* \Delta_I 
\end{equation}

\subsection{Chiodo's Formula}\label{chiodo}

Chiodo's formula states that for the universal $r$th root $\sL$ of $\omega^s_{log}$ on  the universal family of pointed orbicurves $\pi: \uC_{g,k} \rightarrow \sW_{g,k}(\gamma_1,\dots,\gamma_k)$, with local group $\langle \gamma_i\rangle$ of order $m_i$ at the $i$th marked point, we have
$$\ch(R^{\bullet}\pi_* (\sL)) =
\sum_{d\ge 0} \left[\frac{B_{d+1}(s/r)}{(d+1)!}\kappa_d - \sum_{i=1}^k\frac{B_{d+1}(\Theta^{\gamma_i})}{(d+1)!}\psi_i^d + \frac12\sum_{\Gamma_{cut}} \frac{rB_{d+1}(\Theta^{\gamma_+})}{(d+1)!}\tilde{\varrho}_{\Gamma_{cut} *}\left(\sum_{\substack{i+j=d-1\\i,j\ge0}} (-\psi_+)^i\psi_{-}^j\right)\right],
$$
where the second sum is taken over all decorated stable graphs $\Gamma_{cut}$ with one pair of tails labelled $+$ and $-$, respectively, so that once the $+$ and $-$ edges have been glued, we get a single-edged, $n$-pointed, connected, decorated graph of genus $g$ and with additional decoration ($\gamma_+$ and $\gamma_-$) on the internal edge.    Each such graph $\Gamma_{cut}$ has the two cut edges, decorated with group elements $\gamma_+$ and $\gamma_-$, respectively, and the map $\tilde{\varrho}_{\Gamma_{cut}}$ is the corresponding gluing map  $\left(\bMgn{\Gamma_{cut}}{}^{r/s}\right)
\rightarrow \bMgn{g}{k}^{r/s}(\gamma_1,\dots,\gamma_k)$.

In the genus zero case, a choice of $\Gamma_{cut}$ is the same as a a partition $K \sqcup K' = \{1, \ldots, k\}$.  We will sum over all partitions containing the marked point 1, so we will not need to multiply the last sum by $\frac{1}{2}$.

Also, recall from Equation \ref{sMvssWprop} that 
$\tilde{\Delta_I} = \frac{1}{r}st^* \Delta_I$.  So, 
 \[
ch_t(R^{\bullet} \pi_* \mathcal{\sL}) = st^*\left(\sum_{d \geq 0} \left(\frac{B_{d+1}(s/r)}{(d+1)!}\kappa_d
-\sum_{i=1}^n \frac{B_{d+1}(m_i/r)}{(d+1)!}\psi_i^d
+ \sum_{K} \frac{B_{d+1}(\gamma_+^K)}{(d+1)!} (j_K)_* (\gamma_{d-1})\right)t^d\right)
\]

 We can use Lemma \ref{psi_relation} to rewrite $\psi_+$ and $\psi_-$ in terms of boundary divisors of $\bMgn{0}{n_+}$ and $\bMgn{0}{n_-}$.  This will enable us to easily push down these classes to $\bMgn{0}{n}$.  This idea comes from \cite{D4}, and yields the following formulas:
 \begin{equation}\label{psi_plus_minus}
 \begin{array}{lr}
 (j_K)_*(\psi_+) = 0 & \text{ if } |K| \leq 2\\
 (j_K)_*(\psi_+) = 
 \sum_{\{1,a,b\} \subseteq I \subseteq K} \Delta_K \Delta_I + 
  \sum_{1 \in I \subseteq K-\{a,b\}} \Delta_K \Delta_{I\cup K^c}
  & \text{ if } |K| > 2\\
 (j_K)_*(\psi_+) = 0 & \text{ if } n-|K| \leq 2\\
 (j_K)_*(\psi_+) = \sum_{\substack{\emptyset \neq I \subset K^c\\a,b, \notin I}} \Delta_K \Delta_{I \cup K}
  & \text{ if } n-|K| > 2\\
\end{array}
 \end{equation}

 Using the formulas in Equation \ref{psi_plus_minus} and the polynomial defined in Equation \ref{chern_to_char} we can now express $\Lambda$ as a polynomial in $\psi$, $\kappa$ and $\Delta$ classes, 

 \[
 \begin{array}{l}
\Lambda_{g,k}^W(\alpha_1, \ldots, \alpha_k)\\
\qquad = (-1)^D f
\left(
 \left(\frac{B_{2}(s/r)}{(2)!}\kappa_1-\sum_{i=1}^n \frac{B_{2}(m_i/r)}{(2)!}\psi_i
+ \sum_{K} \frac{B_{2}(\gamma_+^K)}{(2)!} (j_K)_* (\psi_- -\psi_+)\right),\ldots\right.\\
 \left.\qquad \left(\frac{B_{D+1}(s/r)}{(D+1)!}\kappa_D-
 \sum_{i=1}^n \frac{B_{D+1}(m_i/r)}{(D+1)!}\psi_i^D
+ \sum_{K} \frac{B_{D+1}(\gamma_+^K)}{(D+1)!} (j_K)_* (\sum_{i + j = D-1} (-\psi_+)^i \psi_-^j)\right)
\right)
\end{array}
\]

The following lemma allows us to always choose a $\gamma_+^K$ in a way that makes sense in Chiodo's formula.

\begin{lem}
Let $B$ be a degree one boundary graph, with decorations $\gamma_1^1, \ldots, \gamma_1^{k_1}$ for the first node and $\gamma_2^1, \ldots, \gamma_2^{k_2}$ for the second,  and genera $g_1$ and $g_2$, respectively.
If a smooth curve with decorations $\gamma_1^1, \ldots, \gamma_1^{k_1}, \gamma_2^1, \ldots, \gamma_2^{k_2}$ and genus $g_1 + g_2$ has integer line bundle degree,
 then it is possible to assign decorations to the edge of $B$ such that each node will have integral line bundle degree.
\end{lem}
\begin{proof}
If the line bundle degree of the smooth curve is integral then:
\[
S_0 = J(2(g_1+g_2) -2 + k_1+k_2) - \sum_i \gamma_1^i -\sum_j \gamma_2^j \in \mathbb{Z}^n
\]
Similarly let $S_1$ and $S_2$ be the equivalent sums in $\QQ^n$, corresponding to the nodes with decorations $\gamma_1^1, \ldots, \gamma_1^{k_1}$ and $\gamma_2^1, \ldots, \gamma_2^{k_2}$, respectively.  
To find $\gamma_0$ we take
\[
\gamma_0 \equiv J(2g_1 -2 + (k_1+1)) - \sum_i \gamma_1^i
\]
Then, by Lemma \ref{lem_lastelement} this will force the $S_1$ to be an integer vector.  

Also,
\[
S_0 = S_1 +\gamma_0+ J(g_2 -2 + (k_2+1)) -\sum_j \gamma_2^j = S_1 + S_2
\]
Which implies that $S_2 \in \mathbb{Z}$.
\end{proof}
We have now described the virtual class $\Lambda$ in the concave case explicitly in terms of tautological classes on the moduli of stable curves.  Now we will talk about how these can actually be computed. 
As a special case we have the following explicit formula for concave genus-zero four-point correlators.
\begin{lem}\label{4pt_concave}
If $\langle \fjrw{1}{g_1}, \fjrw{1}{g_2}, \fjrw{1}{g_3}, \fjrw{1}{g_4} \rangle$ is a genus-zero, four-point correlator which satisfies parts (2) and (3) of Axiom \ref{sel_rules}, and 
if it is also satisfies the hypotheses of  the concavity axiom, (all insertions are narrow, all line bundle degrees are negative, and all line bundle degrees of the nodes of the boundary graphs $\Delta_{1,2}$, $\Delta_{1,3}$, and $\Delta_{1,4}$ are all negative),  then
\[
\left\langle \fjrw{1}{g_1}, \fjrw{1}{g_2}, \fjrw{1}{g_3}, \fjrw{1}{g_4} \right\rangle = \frac{1}{2}\sum_{i=1}^N \left(B_2(q_i) - \sum_{j=1}^4 B_2((g_j)_i) + \sum_{k=1}^3 B_2(\gamma^j_+) \right)
\]
\[
 =\frac{1}{2} \sum_{i=1}^N \left(q_i(q_i-1)- \sum_{j=1}^4 \theta_j^i(\theta_j^i-1) + \sum_{k=1}^3 \gamma^j_+(\gamma^j_+-1) \right)
\]

\noindent where for each $j$ $g_j = (\theta_j^1, \ldots, \theta_j^n)$, $\gamma^1_+ \equiv J - g_1 - g_2$, $\gamma^2_+ \equiv J - g_1 - g_3$, and $\gamma^3_+ \equiv J - g_1 - g_4$, and $B_2$ is the 2nd Bernoulli polynomial.
\end{lem}
\begin{proof}
\[\begin{array}{rl}
\ds{ch_t(-R^1 \pi_* \mathcal{L}_i)=\sum_{d \geq 0} \bigg(\frac{B_{d+1}(q_i)}{(d+1)!}\kappa_d}& \ds{-\sum_{j=1}^4 \frac{B_{d+1}(\theta_j^i)}{(d+1)!}\psi_j^d }\\
	& \ds{+ r\sum_{K} \frac{B_{d+1}((\gamma_+^K)_i)}{(d+1)!} (p_K)_* \sum_{k= 0}^{d-1}(-\psi_+)^k(\psi_-)^{d-1-k}\bigg)t^d}
\end{array}
\]

Which means that 
\[
ch_1 (-R^1 \pi_* \mathcal{L}_i) =\frac{B_{2}(q_i)}{(2)!}\kappa_1
-\sum_{j=1}^4 \frac{B_{2}(\theta_j^i)}{(2)!}\psi_j + r\sum_{K} \frac{B_{2}((\gamma_+)_K)}{(2)!} (p_K)_* (1)
\]

Notice that $(p_K)_*(1_{\bMgn{0}{3}}) = \Delta_K$, and that for $\bMgn{0}{4}$, our choices for $K \in \Gamma_{cut}$ are just $\{1,2\}$, $\{1,3\}$, and $\{1,4\}$.  Numbering these gives: 

\begin{displaymath} 
     \begin{array}{ll}
      (\gamma_+)_1 = J - g_1 - g_2 , & \text{ for }K = \{1,2\},\\
       (\gamma_+)_2 =J - g_1 - g_3 , & \text{ for }K = \{1,3\}, \text{ and }\\
       (\gamma_+)_3 =J - g_1 - g_4 , & \text{ for }K = \{1,4\}.\\
     \end{array}
\end{displaymath} 

Since $B_2(x) = x^2 - x + \frac{1}{6}$, 

\[
\begin{array}{rl}
\ds{ch_1 (-R^1 \pi_* \mathcal{L}_i) }&\ds{=\frac{1}{2}\Bigg(\bigg(q_i(q_i - 1) + \frac{1}{6}\bigg) \kappa_1}\\
&\ds{
-\sum_{j=1}^4\bigg( \theta_j^i(\theta_j^i-1) + \frac{1}{6}\bigg)\psi_j}\\
&\ds{ + r\sum_K \bigg(\gamma_+^K(\gamma_+^K - 1)+\frac{1}{6}\bigg) \tilde{\Delta_{K}} 
   \Bigg).}
\end{array}
\]

The psi and kappa classes in $\sW_{0,4}$ are all pullbacks of the equivalent psi and kappa classes in $\bMgn_{0,4}$, and the $\Delta_I$ classes are scalar multiples of the equivalent classes in $\bMgn_{0,4}$, as in Equation \ref{sMvssWprop}

So, 
\[
\begin{array}{rl}
\ds{ch_1 (-R^1 \pi_* \mathcal{L}_i) }&\ds{=\frac{1}{2}st^*\bigg(\big(q_i(q_i - 1) \big) \kappa_1
-\sum_{j=1}^4\big( \theta_j^i(\theta_j^i-1) \big)\psi_j+ \sum_K\big(\gamma_+^K(\gamma_+^K - 1)\big) \Delta_{K}   \bigg).}
\end{array}
\]

We use Equation \ref{chartoclass} to convert to chern classes, 
\[\begin{array}{rl}
\ds{c_t(-R^1 \pi_* \mathcal{L}_i)}&\ds{ = exp\left(\sum_{i = 1}^\infty (-1)^{i-1} (i-1)! ch_i(-R^1 \pi_* \mathcal{L}_i)t^i\right)
}\\
&\ds{ = \sum_{j = 0}^\infty \frac{1}{j!}\left(\sum_{i = 1}^\infty (-1)^{i-1} (i-1)! ch_t(-R^1 \pi_* \mathcal{L}_i)t_i\right)^j}.
\end{array}\]
Which means that 
$c_0(-R^1 \pi_* \mathcal{L}_i) = 1$ and $c_1(-R^1 \pi_* \mathcal{L}_i) = ch_1(-R^1 \pi_* \mathcal{L}_i)$, then, by 
Equation \ref{chernclasssplit}  since $D = 1$, 

\[
\begin{array}{rl}
\ds{c_1(-R^1 \pi_* \oplus_i \mathcal{L}_i) }&\ds{= \sum_{\substack{0 \leq j_1, \ldots j_N \\ j_1 + \ldots+ j_N = 1 }} \prod c_{j_i}(-R^1 \pi_* \mathcal{L}_i)}
\\
&\ds{ = \sum_{i=1}^N c_1(-R^1 \pi_* \mathcal{L}_i) = \sum_{i=1}^N ch_1(-R^1 \pi_* \mathcal{L}_i).}
\end{array}
\]

Finally we recall that $
 c_t(R^1 \pi_* \mathcal{L}_i) = -\sum_{j} (c_t(-R^1 \pi_* \mathcal{L}_i))^j$, so, 
\[
c_1(R^1 \pi_* \mathcal{L}_i) = -c_1(-R^1 \pi_* \mathcal{L}_i) = -\sum_{i=1}^N ch_1(-R^1 \pi_* \mathcal{L}_i).
\]
And, from Equation \ref{concave_lambda}
\[\begin{array}{c}
\ds{\Lambda_{0,4}(\alpha_1, \ldots, \alpha_4 ) = \frac{1}{deg(st)}PD st_*PD^{-1} (-1)^1 (- \sum_{i=1}^N ch_1(-R^1 \pi_* \mathcal{L}_i))}\\
\ds{= \frac{1}{deg(st)}PD st_*PD^{-1}\sum_{i=1}^N ch_1(-R^1 \pi_* \mathcal{L}_i))}\\
\ds{=\frac{1}{2}\Bigg(\bigg(q_i(q_i - 1) + \frac{1}{6}\bigg) \kappa_1
-\sum_{j=1}^4\bigg( \theta_j^i(\theta_j^i-1) + \frac{1}{6}\bigg)\psi_j + r\sum_K \bigg(\gamma_+^K(\gamma_+^K - 1)+\frac{1}{6}\bigg) \tilde{\Delta_{K}} 
   \Bigg).}

\end{array}
\]

Next, we notice that if $p: \bMgn{0}{4} \rightarrow (\bullet)$ is the map sending all of $\bMgn{0}{4}$ to a point, then the push forward of any of the cohomology classes mentioned above is equal to 1. That is,
\[
p_*\kappa_1 = p_* \psi_i = p_*\Delta_K = 1.
\]
So, 
\[
\begin{array}{l}
\ds{ \left\langle \fjrw{1}{g_1}, \fjrw{1}{g_2}, \fjrw{1}{g_3}, \fjrw{1}{g_4} \right\rangle }
\\
\begin{array}{c}
\ds{=p_* \frac{1}{deg(st)}PD st_*PD^{-1}\sum_{i=1}^N ch_1(-R^1 \pi_* \mathcal{L}_i))}
\\
\ds{ =\frac{1}{2} p_*\sum_{i=1}^N\left(\left(q_i(q_i - 1) \right) \kappa_1
-\sum_{j=1}^4\left( \theta_j^i(\theta_j^i-1) \right)\psi_j  + \sum_K\left(\gamma_+^K(\gamma_+^K - 1)\right) \Delta_{K}\right)}\\
\ds{ =\frac{1}{2}\sum_{i=1}^N\left(q_i(q_i - 1)  + \gamma_+^1(\gamma_+^1 - 1)  + \gamma_+^2(\gamma_+^2 - 1) +\gamma_+^3(\gamma_+^3 - 1)  -\sum_{j=1}^4\ \theta_j^i(\theta_j^i-1)   \right) .}
\end{array}\end{array}\]
\end{proof}

\begin{example}\label{ExampleConcave}
We will compute the genus-zero four-point correlator \\
$\left\langle  \fjrw{1}{(1/4,1/2)}, \fjrw{1}{(1/4,1/2)}, \fjrw{1}{(1/4,1/2)}, \fjrw{1}{(3/4,5/6)} \right\rangle_{0,4}$ for the A-model $\sH_{W_{1,0},\langle J \rangle }$.

It is straightforward to verify that it satisfies Axiom \ref{sel_rules}, and that each of the degree one nodal degenerations of the dual graph are concave.  Thus, 
\[\begin{array}{c}
\left\langle  \fjrw{1}{(1/4,1/2)}, \fjrw{1}{(1/4,1/2)}, \fjrw{1}{(1/4,1/2)}, \fjrw{1}{(3/4,5/6)} \right\rangle_{0,4} = \\
\frac{1}{2} \left(\left(-\frac{3}{16}-\frac{3}{16}-\frac{3}{16}-\frac{3}{16}+\frac{3}{16}+\frac{3}{16}+\frac{3}{16}+\frac{3}{16}\right) \right.
+\left. \left(-\frac{5}{36}-\frac{5}{36}-\frac{5}{36}-\frac{5}{36} +\frac{1}{4}+\frac{1}{4}+\frac{1}{4}+\frac{5}{36}\right) \right)\\
\frac{1}{2}\left( 0 + \frac{1}{3}\right)
=\frac{1}{6}
\end{array}
\]

There are ten concave four-point correlators in $\sH_{W_{1,0},\langle J\rangle}$ and 2 which are not concave.  We will see in the next section that we can use the values of the concave correlators to find other correlator values.

\end{example}
\subsection{Using the Reconstruction Lemma}
In this section we show how to use known correlator values to find unknown correlator values.  In some cases our new methods for computing concave correlators will allow us to compute all genus-zero correlators in the A-model.

The WDVV equations are a powerful tool which can be derived from the Composition axiom.   Applying these equations to correlators, we get the following lemma, 

\begin{lem}\label{Reconstruction}\cite{D4}
\defital{Reconstruction Lemma}.

Any genus-zero, $k$-point correlator of the form 
\[
\left\langle\gamma_1,...,\gamma_{k-3},\alpha,\beta,\epsilon\star\phi\right\rangle_{0,k}
\]
where $ 0 < \deg_\C(\epsilon), \deg_\C(\phi) < \hat{c}$, 
can be rewritten as
 \begin{align*}
\left\langle\gamma_1,...,\gamma_{k-3},\alpha, \beta,\epsilon \star \phi\right\rangle&=
\sum_{I \sqcup J = [k-3]} \sum_l c_{I,J} \left\langle\gamma_{k \in I}, \alpha, \epsilon, \delta_l\right\rangle
		\left\langle\delta'_l, \phi, \beta, \gamma_{j \in J}\right\rangle\\
&-\sum_{\substack{I \sqcup J  = [k-3]\\J \neq \emptyset}}\sum_l c_{I,J}\left\langle\gamma_{k \in I}, \alpha, \beta, \delta_l\right\rangle
		\left\langle \delta'_l, \phi, \epsilon, \gamma_{j \in J} \right\rangle.
\end{align*}
where the $\delta_l$ are the elements of some basis $\mathscr{B}$ and $\delta_l'$ are the corresponding elements of the dual basis $\mathscr{B}'$, and $c_{I,J} = \frac{\prod n_K(\xi_k)!}{\prod n_I(\xi_i)! \prod n_J(\xi_j)!}$. Here $n_X(\xi_x)$ refers to the number of elements equal to $\xi_x$ in the tuple $X$. The product $\prod n_X(\xi_x)! $ is taken over all distinct elements $\xi_x$ in $X$.
\end{lem}

We say that an element $\alpha \in \sH_{W,G}$ is \defital{non-primitive} if it can be written $\epsilon \star \phi = \alpha$ for some $\epsilon$ and $\phi$ in $\sH_{W,G}$ with 
$0 < \deg_{\CC} \epsilon, \deg_{\CC} \phi < \deg_{\CC} \alpha$.  Otherwise, we say that $\alpha$ is \defital{primitive}.

\begin{cor}
A genus-zero, four-point correlator containing a \textit{non-primitive} insertion can be rewritten:
 \begin{align*}
\left\langle\gamma,\alpha, \beta,\epsilon \star \psi\right\rangle_{0.4}=
\sum_l \left\langle\gamma, \alpha, \epsilon, \delta_l\right\rangle&\left\langle\delta'_l, \psi, \beta\right\rangle
+ \sum_l \left\langle \alpha, \epsilon, \delta_l\right\rangle\left\langle\delta'_l, \psi, \beta,\gamma\right\rangle\\
&-\sum_l \left\langle \alpha, \beta, \delta_l\right\rangle
		\left\langle \delta'_l, \psi, \epsilon, \gamma \right\rangle.
\end{align*}

\end{cor}

In fact, using the reconstruction lemma, it is possible to write any genus-zero $k$-point correlator in terms of the pairing, genus-zero three-point correlators and correlators of the form $\langle \gamma_1, \ldots, \gamma_k' \rangle$ for $k' \leq k$ where $\gamma_i$ is primitive for $i \leq k' - 2$ (see \cite{FJR1}).

We say that a correlator is \defital{basic} if at most two of the insertions are non-primitive. 

Recall that the \defital{central charge} of the polynomial $W$ is 
$
\hat{c} = \sum_j (1 - 2q_j).
$

\begin{lem}[\cite{FJR1}]\label{all_corrs_lem}
 If $\deg_{\CC}(\alpha) < \hat{c}$ for all classes $\alpha$, $P$ is the maximum complex degree of any primitive class, and $P <1$, then all the genus-zero correlators are uniquely determined by the pairing, the genus-zero three-point correlators,  and the basic genus-zero $k$-point correlators for each $k$ that satisfies
\begin{equation}\label{klimit}
 k \leq 2 + \frac{1 + \hat{c}}{  1-P}.
\end{equation}
\end{lem}

\begin{example}
In the case of $\sH_{W_{1,0},\langle J \rangle}$, the maximal primitive degree is $7/12$ and $\hat{c} = 7/6$, so the Frobenius manifold structure is determined by the genus-zero three-point correlators and the basic genus-zero $k$-point correlators for $k \leq 7$. 

 All of the three-point correlators can be found using the methods described in \cite{KPABR}.  
There are 165 distinct genus-zero three point correlators.  The selection rules show that 158 of them vanish.  All of the seven remaining correlator values can be found using a pairing axiom found in \cite{FJR}.  
 There are ten concave and two nonconcave four-point correlators which satisfy Axiom \ref{sel_rules}, which might therefore be nonzero.  We can use the values of the concave correlators to find values of the others using Lemma \ref{Reconstruction}.  For example, if we let 
\[
 X = \fjrw{1}{(1/4,1/2)}, Y = \fjrw{1}{(1/2,1/3)}, Z = \fjrw{1}{(3/4,1/6)},\text{ and }W = \fjrw{xy^2}{(0,0)},
 \]
  then the nonconcave four point correlators are $\langle X, W,W,X^2\rangle$ and $\langle Z, Z, W, W \rangle$.  Lemma \ref{Reconstruction} gives
\[
\langle X,W,W,X^2 \rangle  =
\ds{\sum_{\delta}2\langle X, W, X, \delta \rangle \langle \delta^\prime, W, X\rangle
 -\langle  W, W, \delta \rangle \langle \delta^\prime, X, X, X\rangle}.
 \]
 The selection rules, methods in \cite{KPABR}, and computations from Example \ref{ExampleConcave} reduce this to 
 \[
 \langle X,W,W,X^2 \rangle  
 =-\langle W,W, \one\rangle \langle XY^2, X, X, X \rangle
 \ds{ = -\frac{1}{24} \cdot \frac{1}{6} = -\frac{1}{144}}.
\]

To find the value of $\langle Z, Z, W, W \rangle$ we have to look for ways to reconstruct it using five-point correlators.  
\[\begin{array}{rl}
\langle Y,W,W,XY, XZ \rangle  &=
\langle X,Y,W,XZ\rangle \langle X,Z,W,XY\rangle + \langle X,W,W,X^2 \rangle \langle Y,Z,Z,XY\rangle \\
&\quad -\langle W,XY,XZ\rangle\langle X,X,Y,Z,W\rangle - \langle Y,Y,W,XY\rangle \langle X,Z,W,XY\rangle.
 \end{array}
\]
The correlators $\langle X, Z, W, XY \rangle$ and $\langle W,XY,XZ\rangle$ are both equal to zero, which can be shown using the methods in \cite{KPABR}.  
The correlators $\langle Y,Z,Z,XY \rangle$, and $\langle X,W,W,X^2 \rangle$ are both concave.  Using Lemma \ref{4pt_concave}, it is a straightforward calculation to find their values, which are  $\frac{1}{6}$ and $\frac{1}{4}$, respectively.
  So we have 
\begin{equation}\label{ex_reconst1}
\langle Y,W,W,XY, XZ \rangle  =\langle X,W,W,X^2 \rangle \langle Y,Z,Z,XY\rangle = -\frac{1}{144} \cdot \frac{1}{4} = -\frac{1}{576}.
\end{equation}

We can reconstruct this same correlator in a different way:
\begin{equation}\label{ex_reconst2}
\begin{array}{rl}
\langle W,Y,XY,W, XZ \rangle  &=
\langle X,Y,W,XZ\rangle \langle Y,Z,W,X^2\rangle + \langle X,Y,X^2,XY \rangle \langle Z,Z,W,W\rangle \\
&\qquad -\langle W,XY,XZ\rangle\langle X,X,Y,Z,W\rangle - \langle Y,Y,W,XY\rangle \langle X,Z,W,XY\rangle\\
&=\langle X,Y,X^2,XY \rangle \langle Z,Z,W,W\rangle = \frac{1}{6}\langle Z,Z,W,W\rangle.  \\
 \end{array}
\end{equation}

Together, Equations \ref{ex_reconst1} and \ref{ex_reconst2} tell us that  $\langle Z,Z,W,W\rangle  =  -\frac{1}{96}$.

To compute the full Frobenius manifold structure of $\sH_{W_{1,0}, \langle J \rangle}$, we still need to find the values of all basic  five, six, and seven-point correlators.  There are fifteen basic five-point correlators, seven of which are not concave.  There are five basic six-point correlators, three of which are not concave, and there is one basic seven-point correlator, which is not concave.  All of these correlators can be computed using the methods established here. Their values are given in Section \ref{sec_done}.
\end{example}

\begin{remark}In certain examples, if we use Lemma \ref{Reconstruction} together with the computed values of concave genus-zero higher point correlators, we can find values of previously unknown three-point correlators. This allows us to find even Frobenius algebra structures that were previously unknown.  
\end{remark}

\section{Summary of Computations}
We compute the FJRW theories with all possible admissible symmetry groups coming from the quasihomogeneous polynomial representatives corresponding to an isolated singularity in Arnol'd's \cite{Ar} list of singularities.  Recall that there are potentially several polynomial representatives corresponding to a singularity.
We note here that the polynomial representatives of $X_9$ and $J_{10}$ which appear in this list of singularities \cite{Ar} were not considered in \cite{shenkrawitz} for any choice of symmetry group.  The non-maximal symmetry groups for $P_8$, also, have not yet been treated.

\subsection{ FJRW Theories which split into known tensor products}\label{sec_splits}
Recall that if a polynomial $W$ can be written $W = W_1+W_2$, where $W_1 \in \C[x_1, \ldots, x_l]$ and $W_2 \in \C[x_{l+1}, \ldots, x_n]$, and if $G$ can be written $G = G_1 \oplus G_2$, where for each $i$, $G_i$ is an admissible symmetry group for $W_i$, then then the A-model associated to the pair $(W,G)$ splits into the tensor product of the A-models for the pairs $(W_1, G_1)$ and $(W_2,G_2)$.  

The ADE singularities (except for $D_4$) with all possible symmetry groups were computed in \cite{FJR}.  The $D_4$ case was done in \cite{D4}.  For the following pairs of polynomial and group, the A-model structure splits into a product of known FJRW theories.  We use the notation $\sA_{W,G}$ to denote the full FJRW theory associated to the pair $(W,G)$. Symmetry groups which are maximal are denoted with the subscript $max$.

{\small

\begin{center}
\begin{tabular}{| c | l | c | c | c |}
\hline
$W$ & $G$ & dimension & A-models  \\\hline
$X_9 = x^4 + y^4$ &  $\langle (1/4,0),(0,1/4)\rangle_{max}$ &9 &  $\sA_{A_3,J}\otimes\sA_{A_3,J} $ \\
$J_{10} = x^3+y^6$ &  $\langle (1/3,0),(0,1/6)\rangle_{max}$ &10 &  $ \sA_{A_2,J} \otimes \sA_{A_5,J}$ \\
$Q_{12}= x^3+y^5+yz^2$ &  $\langle (1/3,0,0),(0,1/5,2/5)\rangle$ &12 &  $\sA_{A_2,J} \otimes \sA_{D_6,J}$ \\
$J_{3,0} = x^3+y^9$ & $ \langle (1/3,0),(0,1/9)\rangle_{max} $&16 &  $\sA_{A_2,J} \otimes \sA_{A_8,J}$ \\
$W_{1,0} = x^4+y^6$ &  $\langle (1/4,0),(0,1/6)\rangle_{max}$ &15 &  $\sA_{A_3,J} \otimes \sA_{A_5,J}$ \\
$E_{18} = x^3+y^{10}$ &  $\langle (1/3,0),(0,1/10)\rangle_{max}$ &18 &  $\sA_{A_2,J} \otimes \sA_{A_9,J}$ \\
$E_{20} = x^3+y^{11}$ &  $\langle (1/3,0),(0,1/11)\rangle_{max}$ &20 &  $\sA_{A_2,J} \otimes \sA_{A_{10},J}$ \\
$U_{16} = x^3+xz^2+y^5$ &  $\langle (1/3,0,5/6),(0,1/5,0)\rangle_{max}$ &20 &  $\sA_{D_4,G^{max}} \otimes \sA_{A_4,J}$ \\
$U_{16}= x^3+xz^2+y^5$ &  $\langle (1/3,0,1/3),(0,1/5,0)\rangle$ &16 &  $\sA_{D_4,J} \otimes \sA_{A_4,J}$ \\
$U_{16}^T= x^3z+z^2+y^5$ &  $\langle (1/6,0,1/2),(0,1/5,0)\rangle_{max}$ &16 &  $\sA_{D_4^T,J}\otimes \sA_{A_4,J}$ \\
$W_{18} = x^4+y^7$ &  $\langle (1/4,0),(0,1/7)\rangle_{max}$ &18 &  $\sA_{A_3,J}\otimes \sA_{A_6,J}$ \\
$Q_{16} = x^3+yz^2+y^7$ &  $ \langle (1/3,0,0),(0,1/7,13/14)\rangle_{max}$ &26 &  $\sA_{A_2,J}\otimes \sA_{D_8,G^{max}}$ \\
$Q_{16}= x^3+yz^2+y^7$ &  $ \langle (1/3,0,0),(0, 1/7,3/7)\rangle$ &16 &  $\sA_{A_2,J}\otimes \sA_{D_8,J}$ \\
$Q_{16}^T= x^3+z^2+y^7z$ &  $\langle (1/3,0,0),(0,1/14,1/2)\rangle_{max}$ &16 &  $\sA_{A_2,J}\otimes \sA_{D_8^T,J}$ \\
$Q_{18}= x^3+yz^2+y^8$ &  $\langle (1/3,0,0),(0,1/8,7/16)\rangle_{max}$ &18 &  $\sA_{A_2,J} \otimes \sA_{D_9,J}$ \\
\hline
\end{tabular} 
\end{center}

}

\subsection{Previously Unknown FJRW Theories we can compute}\label{sec_done}
Recall that in Lemma \ref{all_corrs_lem} we saw that the Frobenius manifold structure can be found from the genus-zero three-point correlators, together with the genus-zero $k$-point correlators for $k$ as in Equation \ref{klimit}.   The value of $k$ depends on the central charge $\hat{c}$ of the polynomial and on the maximum degree $P$ of any primitive element in the state space.  The value of $k$ is given for each of the A-models below, together with all necessary correlators.
Correlators marked with a $\ast$ must be found using the Reconstruction Lemma.  Unmarked $k$-point correlators for $k >3$ can be found using the concavity axiom.  

{\small 

{\centering

\noindent\begin{tabular}{||c|c|c||}
\hline\hline
\multicolumn{3}{||c||}{$P_8 = x^3+y^3+z^3, \quad G = \langle \gamma = (1/3,0,2/3), \gamma_2 = (0,1/3,2/3)\rangle$} \\\hline\hline
\multicolumn{3}{||c||}{$ k \leq 6$, $X = e_{0}, Y =xyze_{0}$. $P = 1/2, \hat{c} = 1$}\\
\multicolumn{3}{||c||}{Relations: $X^2=Y^2 = 0$;  $\dim = 4$ }\\\hline
 	\begin{tabular}{c}
 	3 pt correlators
 	\\ \hline
 	$\langle \mb{1}, \mb{1}, XY \rangle = 1/27$\\
 	$\langle \mb{1}, X, Y \rangle = 1/27$\\
	 \end{tabular}
 &
		\begin{tabular}{c}
		4 pt correlators
	 	\\ \hline
		None. \\
		\\\hline
		5pt correlators \\\hline
		None. 
		\end{tabular}
 &
 		\begin{tabular}{c}
		6 pt correlators
 		\\ \hline
		$\langle X,X,Y,Y,XY,XY \rangle_\ast = 0$
		\end{tabular}\\\hline \hline
\end{tabular}

\noindent\begin{tabular}{||c|c|c||}
\hline\hline
\multicolumn{3}{||c||}{$P_8 = x^3+y^3+z^3, \quad G = \langle J = (1/3,1/3,1/3)\rangle$} \\\hline\hline
\multicolumn{3}{||c||}{$k \leq 6 $, $X = e_{0}, Y =xyze_{0} $. $P = 1/2, \hat{c} = 1$ }\\
\multicolumn{3}{||c||}{Relations: $X^2=Y^2 = 0$;  $\dim = 4$ }\\\hline
\hline
 	\begin{tabular}{c}
 	3 pt correlators
 	\\ \hline
 	$\langle \mb{1}, \mb{1}, XY \rangle = 1/27$\\
 	$\langle \mb{1}, X, Y \rangle = 1/27$\\
	 \end{tabular}
 &
		\begin{tabular}{c}
		4 pt correlators
	 	\\ \hline
		None. \\
		\\\hline
		5pt correlators \\\hline
		None. 
		\end{tabular}
 &
 		\begin{tabular}{c}
		6 pt correlators
 		\\ \hline
		None.
		\end{tabular}\\\hline \hline
\end{tabular}

\noindent\begin{tabular}{||c|c|c||}
\hline\hline
\multicolumn{3}{||c||}{$X_9 = x^4 + y^4, \quad G = \langle J = (1/4,1/4), \gamma = (0,1/2)\rangle$} \\\hline\hline
\multicolumn{3}{||c||}{$k \leq 6 $, $X = xye_{0}, Y =e_{J+\gamma}, Z = e_{2J},W =  e_{3J+\gamma} $. $P = 1/2, \hat{c} = 1$ }\\
\multicolumn{3}{||c||}{Relations:  $16X^2 = YW = Z^2, X^3 = Z^3 = Y^2 = W^2 = 0$. $\dim = 6$  }\\
\hline
 	\begin{tabular}{c}
 	3 pt correlators
 	\\ \hline
 	$\langle \mb{1}, \mb{1}, Z^2 \rangle = 1$\\
 	$\langle \mb{1}, Y,W\rangle = 1$\\
 	$\langle \mb{1}, Z,Z \rangle = 1$\\
 	$\langle \mb{1}, X,X \rangle = \frac{1}{16}$\\\\
		4 pt correlators
	 	\\ \hline
		$\langle Y,Y,Y,W \rangle = 0$\\
		$\langle Y,Y,Z,Z \rangle = 1/4$\\
		$\langle Y,W,W,W \rangle = 0$\\
		$\langle Z,Z,W,W \rangle =1/4$\\
		$\langle X,X,Y,Y \rangle_\ast =-1/64$\\
	 \end{tabular}

 &
		\begin{tabular}{c}
		5 pt correlators
	 	\\ \hline
		$\langle Y,Y,Y,Y,Z^2 \rangle = \frac{1}{8}$\\
		$\langle Y,Y,W,W,Z^2 \rangle =0$\\
		$\langle Y,Z,Z,W,W^2 \rangle=  \frac{1}{16}$\\
		$\langle Z,Z,Z,Z,W^2 \rangle = \frac{1}{16}$\\
		$\langle W,W,W,W,W^2 \rangle = \frac{1}{8}$\\
		$\langle X,X,X,X,W^2 \rangle_\ast = \frac{1}{4096}$\\
		$\langle X,X,Y,W,W^2 \rangle_\ast = \frac{1}{256}$\\
		$\langle X,X,Z,Z,W^2 \rangle_\ast = -\frac{1}{256}$\\\\\\\\
		\end{tabular}
 &
 		\begin{tabular}{c}
		6 pt correlators
 		\\ \hline
		$\langle Y,Y,Y,W,W^2,W^2 \rangle = 0$\\
		$\langle Y,Y,Z,Z,W^2,W^2 \rangle = \frac{1}{32}$\\
		$\langle Y,W,W,W,W^2,W^2 \rangle = 0$\\
		$\langle Z,Z,W,W,W^2,W^2 \rangle = \frac{1}{32}$\\
		$\langle X,X,Y,Y,W^2,W^2 \rangle_\ast = -\frac{1}{512}$\\
		$\langle X,X,Z,Z,W^2,W^2 \rangle_\ast = -\frac{1}{512}$\\\\\\\\\\\\
		\end{tabular}\\\hline \hline
\end{tabular}

\noindent\begin{tabular}{||c|c|c||}
\hline\hline
\multicolumn{3}{||c||}{$X_9 = x^4 + y^4, \quad G = \langle J = (1/4,1/4)\rangle$} \\\hline\hline
\multicolumn{3}{||c||}{$k \leq 6 $, $X = y^2e_{0}, Y =xye_{0}, Z = x^2 e_0, W=e_{2J} $. $P = 1/2, \hat{c} = 1$ }\\
\multicolumn{3}{||c||}{Relations:  $16XZ = 16Y^2 = W^2, X^2 = Z^2 = Y^3=W^3=0$. $\dim = 6$  }\\
\hline
 	\begin{tabular}{c}
 	3 pt correlators
 	\\ \hline
 	$\langle \mb{1}, \mb{1}, W^2 \rangle = 1$\\
 	$\langle \mb{1}, W,W \rangle = 1$\\
 	$\langle \mb{1}, X,Z \rangle = \frac{1}{16}$\\
 	$\langle \mb{1}, Y,Y \rangle = \frac{1}{16}$\\\\
		4 pt correlators
	 	\\ \hline
		$\langle X,X,Y,W \rangle_\ast = a\neq 0$\\
		$\langle Y,Z,Z,W \rangle_\ast =-\frac{1}{8192a}$\\\\
	 \end{tabular}

 &
		\begin{tabular}{c}
		5 pt correlators
	 	\\ \hline
		$\langle W,W,W,W,W^2 \rangle = \frac{1}{16}$\\
		$\langle X,X,X,X,W^2 \rangle_\ast =32a^2$\\
		$\langle X,X,Z,Z,W^2 \rangle_\ast =  0$\\
		$\langle X,Y,Y,Z,W^2 \rangle_\ast = -\frac{1}{4096}$\\
		$\langle X,Z,W,W,W^2 \rangle_\ast = -\frac{1}{256}$\\
		$\langle Y,Y,Y,Y,W^2 \rangle_\ast = \frac{1}{4096}$\\
		$\langle Y,Y,W,W,W^2 \rangle_\ast = -\frac{1}{256}$\\
		$\langle Z,Z,Z,Z,W^2 \rangle_\ast = \frac{1}{134217728a^2}$\\\\
		\end{tabular}
 &
 		\begin{tabular}{c}
		6 pt correlators
 		\\ \hline
		$\langle X,X,Y,W,W^2,W^2 \rangle_\ast = -\frac{a}{8}$\\
		$\langle Y,Z,Z,W,W^2,W^2 \rangle_\ast = \frac{1}{65536a}$\\\\\\\\\\\\\\\\
		\end{tabular}\\\hline \hline
\end{tabular}


 \begin{tabular}{||c|c||}
\hline\hline
\multicolumn{2}{||c||}{$J_{10}=x^3+y^6, \quad G=\langle J = (1/3,1/6)\rangle$} \\\hline\hline
\multicolumn{2}{||c||}{$k \leq 6 $, $X = y^3e_{0}, Y =xye_{0}, Z = e_{4J},W =  e_{2J} $. $P = 1/2, \hat{c} = 1$ }\\
\multicolumn{2}{||c||}{Relations:  $18XY = ZW, X^2=Y^2=Z^2=W^2 = 1$, $\dim = 6$  }\\
\hline
	\begin{tabular}{c}
 	3 pt correlators
 	\\ \hline
 	$\langle \mb{1}, \mb{1}, ZW \rangle = 1$\\
 	$\langle \mb{1}, Z,W\rangle = 1$\\
 	$\langle \mb{1}, X,Y \rangle = \frac{1}{18}$\\\\
		4 pt correlators
	 	\\ \hline
		$\langle Z,Z,Z,W \rangle = 1/6$\\
		$\langle W,W,W,W \rangle = 1/3$\\
		$\langle X,X,X,W \rangle_\ast = a \neq 0$\\
		$\langle X,Y,Z,Z \rangle_\ast =-\frac{1}{108}$\\
		$\langle Y,Y,Y,W \rangle_\ast =-\frac{1}{104,976a}$\\
	 \end{tabular}

 &
		\begin{tabular}{c}
		5 pt correlators
	 	\\ \hline
		$\langle Y,Y,Y,Y,ZW \rangle = \frac{1}{9}$\\
		$\langle Y,Z,Z,Z,ZW \rangle =\frac{1}{18}$\\
		$\langle X,X,X,Z,ZW \rangle_\ast=  -\frac{a}{3}$\\
		$\langle X,Y,W,W,ZW \rangle_\ast = -\frac{1}{324}$\\
		$\langle Y,Y,Y,Z,ZW \rangle_\ast = \frac{1}{314,928a}$\\\\
		6 pt correlators
 		\\ \hline
		$\langle Z,Z,W,W,ZW,ZW \rangle = \frac{1}{54}$\\
		$\langle X,X,Y,Y,ZW,ZW \rangle_\ast = \frac{1}{17,496}$\\
		$\langle X,Y,Z,W,ZW,ZW \rangle_\ast = \frac{1}{1944}$\\
		\end{tabular}\\\hline \hline
\end{tabular}

\begin{tabular}{||c|c||}
\hline\hline
\multicolumn{2}{||c||}{$Z_{13}^T=x^3+xy^4, \quad G = \langle J = (1/3,1/9)\rangle$} \\\hline\hline
\multicolumn{2}{||c||}{$k \leq 6 $, $X = e_{4J}, Y =e_{2J}, Z = y^5e_0, W = xy^2e_0$. $P = 5/9, \hat{c} = 10/9$. }\\
\multicolumn{2}{||c||}{Relations: $X^2Y=-6Z^2=18W^2, X^3 = Y^2 = Z^3 = W^3 = 0$, $\dim = 8$.  }\\
\hline
\begin{tabular}{c}
{3 pt correlators}
\\ \hline
\begin{tabular}{cc}
$\langle \one, \one, X^2Y\rangle= 1$&
$\langle \one, X, XY\rangle= 1$\\ 
$\langle \one, Y, X^2\rangle= 1$&
$\langle X, X, Y\rangle= 1$\\ 
$\langle \one, Z, Z\rangle= -1/6$&
$\langle \one, W, W\rangle= 1/18$\\ 
\end{tabular}\\\hline
{5 pt correlators}
\\ \hline
{$\langle X,Y,Y,XY,X^2Y \rangle=1/27$}\\
{$\langle X,X,Z,XY,X^2Y \rangle=1/27$}\\
{$\langle X,Y,Z,X^2,X^2Y \rangle_\ast=0$}\\
{$\langle X,Z,Z,Z,X^2Y \rangle_\ast=0$}\\
{$\langle X,Z,W,W,X^2Y \rangle_\ast=-2a/9$}\\
{$\langle Y,Y,Y,X^2,X^2Y \rangle_\ast=-a/162$}\\
{$\langle Y,Y,Z,Z,X^2Y \rangle_\ast=1/162$}\\
{$\langle Y,Y,W,W,X^2Y \rangle_\ast=-1/486$}\\
{$\langle Y,Z,Z,XY,XY \rangle_\ast=1/162$}\\
{$\langle Y,W,W,XY,XY \rangle_\ast=-1/486$}\\
{$\langle Z,Z,Z,X^2,XY \rangle_\ast=-2a/9$}\\
{$\langle Z,W,W,X^2,XY \rangle_\ast=-a/162$}\\\\
\end{tabular}&
\begin{tabular}{c}
4 pt correlators
\\ \hline
$\langle X,X,X,X^2Y \rangle = 1/9$\\
$\langle X,X,X^2,XY \rangle = 2/9$\\
$\langle X,Y,X^2,X^2 \rangle = 1/9$\\
$\langle Y,Y,Y,XY \rangle =1/3$\\
$\langle X,Y,Z,XY \rangle_\ast =0$\\
$\langle X,Z,Z,X^2 \rangle_\ast =1/54$\\
$\langle X,W,W,X^2 \rangle_\ast =-1/162$\\
$\langle Y,Y,Z,X^2 \rangle_\ast =0$\\
$\langle X,Z,Z,Z \rangle_\ast =a = \pm 1/9$\\
$\langle X,Z,W,W \rangle_\ast =a/36$\\\hline
{6 pt correlators}
\\ \hline
{$\langle X,Y,Z,Z,X^2Y,X^2Y \rangle_\ast=-1/1458$}\\
{$\langle X,Y,W,W,X^2Y,X^2Y \rangle_\ast=1/4374$}\\
{$\langle Y,Y,Y,Z,X^2Y,X^2Y \rangle_\ast=0$}\\
{$\langle Z,Z,Z,Z,XY,X^2Y \rangle_\ast=1/2916$}\\
{$\langle Z,Z,W,W,XY,X^2Y \rangle_\ast=-1/26244$}\\
{$\langle W,W,W,W,XY,X^2Y \rangle_\ast=1/26244$}
\end{tabular}\\
\hline\hline
\end{tabular}

\begin{tabular}{||c|c||}
\hline\hline
\multicolumn{2}{||c||}{$J_{3,0}=x^3+y^9, \quad G =\langle J = (1/3,1/9)\rangle$} \\\hline\hline
\multicolumn{2}{||c||}{$k \leq 6 $, $X = e_{4J}, Y =e_{2J}, Z = y^5e_0, W = xy^2e_0$. $P = 5/9, \hat{c} = 10/9$. }\\
\multicolumn{2}{||c||}{Relations:  $27ZW = X^2Y $, $ab = -1/531441$, $X^3 = Y^2=Z^2 = W^2 = 0$, $\dim = 8$.  }\\
\hline
 	\begin{tabular}{c}
{3 pt correlators}
 	\\ \hline
	\begin{tabular}{cc}
$\langle \one, \one, X^2Y\rangle= 1$&
$\langle \one, X, XY\rangle= 1$\\ 
$\langle \one, Y, X^2\rangle= 1$&
$\langle \one, Z, W\rangle= 1/27$\\ 
$\langle X, X, Y\rangle= 1$&
\end{tabular}\\\\
4 pt correlators
	 	\\ \hline
		$\langle X,X,X,X^2Y \rangle = 1/9$\\
		$\langle X,X,X^2,XY \rangle = 2/9$\\
		$\langle X,Y,X^2,X^2 \rangle = 1/9$\\
		$\langle X,Z,W,X^2 \rangle_\ast = -1/243$\\
		$\langle Y,Y,Y,XY \rangle =1/3$\\
		$\langle Y,Z,Z,Z \rangle =a$\\
		$\langle Y,W,W,W \rangle =b$\\\\\hline
\end{tabular}&
\begin{tabular}{c}
		5 pt correlators
	 	\\ \hline
		$\langle X,Y,Y,XY,X^2Y \rangle = 1/27$\\
		$\langle X,Z,Z,Z,X^2Y \rangle_\ast = -2a/9$\\
		$\langle X,W,W,W,X^2Y \rangle_\ast = -2b/9$\\
		$\langle Y,Y,Y,X^2,X^2Y \rangle =1/27$\\
		$\langle Y,Y,Z,W,X^2Y \rangle_\ast =-1/729$\\
		$\langle Y,Z,W,XY,XY \rangle_\ast =-1/729$\\
		$\langle Z,Z,Z,X^2,XY \rangle_\ast =-2a/9$\\
		$\langle W,W,W,X^2,XY \rangle_\ast =-2b/9$\\
		6 pt correlators
	 	\\ \hline
		$\langle X,Y,Z,W,X^2Y,X^2Y \rangle_\ast = 1/6561$\\
		$\langle Z,Z,W,W,XY,X^2Y \rangle_\ast = 4/177147$\\

		\\\end{tabular}\\
\hline\hline
\end{tabular}

\begin{tabular}{||c|c||}
\hline\hline
\multicolumn{2}{||c||}{$Z_{1,0}=x^3y+y^7, \quad G^{max} = \langle  \gamma = (1/21,6/7)\rangle$} \\\hline\hline
\multicolumn{2}{||c||}{$k \leq 5 $, $X = e_{5\gamma}, Y =e_{13\gamma}$. $P = 1/3, \hat{c} = 8/7$. }\\
\multicolumn{2}{||c||}{Relations:  $X^7 = -3Y^2, X^{13} = Y^3 = 0 $, $\dim = 19$.  }\\
\hline
 	\begin{tabular}{cc}
 	\multicolumn{2}{c}{3 pt correlators}
 	\\ \hline
$\langle \one, \one, X^5Y^2\rangle= 1$&
$\langle \one, X, X^4Y^2\rangle= 1$\\ 
$\langle \one, X^2, X^3Y^2\rangle= 1$&
$\langle \one, X^3, X^2Y^2\rangle= 1$\\ 
$\langle \one, Y, X^5Y\rangle= 1$&
$\langle \one, X^4, XY^2\rangle= 1$\\ 
$\langle \one, XY, X^4Y\rangle= 1$&
$\langle \one, X^5, Y^2\rangle= 1$\\ 
$\langle \one, X^2Y, X^3Y\rangle= 1$&
$\langle \one, X^6, X^6\rangle= -3$\\ 
$\langle X, X, X^3Y^2\rangle= 1$&
$\langle X, X^2, X^2Y^2\rangle= 1$\\ 
$\langle X, X^3, XY^2\rangle= 1$&
$\langle X, Y, X^4Y\rangle= 1$\\ 
$\langle X, X^4, Y^2\rangle= 1$&
$\langle X, XY, X^3Y\rangle= 1$\\
$\langle X, XY, X^2Y\rangle= 1$& 
$\langle X, X^5, X^6\rangle= -3$\\ 
$\langle X, X^2Y, X^2Y\rangle= 1$& 
$\langle X^2, X^2, XY^2\rangle= 1$\\ 
$\langle X^2, X^3, Y^2\rangle= 1$& 
$\langle X^2, Y, X^3Y\rangle= 1$\\
$\langle X^2, X^4, X^6\rangle= -3$&
$\langle X^2, XY, X^2Y\rangle= 1$\\
$\langle X^2, X^5, X^5\rangle_\ast= -3$&
$\langle X^3, X^3, X^6\rangle= -3$\\
$\langle X^3, Y, X^2Y\rangle= 1$&
$\langle X^3, X^4, X^5\rangle_\ast= -3$\\
$\langle X^3, XY, XY\rangle= 1$&
$\langle Y, Y, X^5\rangle= 1$\\  
$\langle Y,  X^4,XY\rangle= 1$&
$\langle X^4,  X^4,X^4\rangle_\ast= -3$\\  
\end{tabular}&
\begin{tabular}{c}
4 pt correlators
	 	\\ \hline
		$\langle X,X,X^5Y,X^5Y^2 \rangle = 1/7$\\
		$\langle X,Y,X^6,X^5Y^2 \rangle_\ast = 1/7$\\
		$\langle X,Y,Y^2,X^4Y^2 \rangle = -1/21$\\
		$\langle X,Y,XY^2,X^3Y^2 \rangle =-1/21$\\
		$\langle X,Y,X^2Y^2,X^2Y^2 \rangle = -1/21$\\
		$\langle Y,Y,Y,X^5Y^2 \rangle =2/7$\\
		$\langle Y,Y,XY,X^4Y^2 \rangle =5/21$\\
		$\langle Y,Y,X^2Y,X^3Y^2 \rangle =4/21$\\
		$\langle Y,Y,X^3Y,X^2Y^2 \rangle =1/7$\\
		$\langle Y,Y,X^4Y,XY^2 \rangle =2/21$\\
		$\langle Y,Y,Y^2,X^5Y \rangle =1/21$\\\\\hline
		5 pt correlators
	 	\\ \hline
		None\\
\end{tabular}\\
\hline\hline
\end{tabular}

\begin{tabular}{||c|c||}
\hline\hline
\multicolumn{2}{||c||}{$Z_{1,0}=x^3y+y^7, \quad G =\langle J = (2/7,1/7)\rangle$} \\\hline\hline
\multicolumn{2}{||c||}{$k \leq 7 $, $X = e_{4J}, Y =e_{2J}, Z = y^4e_0, W = xy^2e_0$. 
$P = 4/7, \hat{c} = 8/7$. }\\
\multicolumn{2}{||c||}{Relations:  $X^3 = -3Y^2 $, $21ZW = XY^2$, $a \neq 0$, $X^5 = Y^3 = Z^2 = W^2 = 0$, $\dim = 9$.  }\\
\hline
 	\begin{tabular}{c}
 	{3 pt correlators}
 	\\ \hline
	\begin{tabular}{cc}
$\langle \one, \one, XY^2\rangle= 1$&
$\langle \one, X, Y^2\rangle= 1$\\ 
$\langle \one, Y, XY\rangle= 1$&
$\langle \one, Z, W\rangle= 1/21$\\ 
$\langle \one, X^2, X^2\rangle= -3$&
$\langle X, Y, Y\rangle= 1$\\
$\langle X, X, X^2\rangle= -3$\\ 
\end{tabular}\\\\\\
4 pt correlators
	 	\\ \hline
		$\langle X,X,Y,XY^2 \rangle = 1/7$\\
		$\langle X,X,XY,Y^2 \rangle = 2/7$\\
		$\langle X,Y,XY,XY \rangle = 1/7$\\
		$\langle X,Y,X^2,Y^2 \rangle_\ast = 3/7$\\
		$\langle X,Z,W,XY \rangle_\ast = -1/147$\\
		$\langle Y,Y,Y,Y^2 \rangle =1/7$\\
		$\langle Y,Y,X^2,XY \rangle_\ast =0$\\
		$\langle Y,Z,Z,Z\rangle_\ast =a$\\
		$\langle Y,Z,W,X^2 \rangle_\ast =-1/147$\\
		$\langle Y,W,W,W \rangle_\ast =-1/(194,481a)$\\
\end{tabular}&
\begin{tabular}{c}
		5 pt correlators
	 	\\ \hline
		$\langle X,X,X,XY^2,XY^2 \rangle = 2/49$\\
		$\langle X,Y,Y,Y^2,XY^2 \rangle = 0$\\
		$\langle X,Z,Z,Z,XY^2 \rangle_\ast = -2a/7$\\
		$\langle X,Z,W,X^2,XY^2 \rangle_\ast = 1/3087$\\
		$\langle X,Z,W,Y^2,Y^2 \rangle_\ast = 1/1029$\\
		$\langle X,W,W,W,XY^2 \rangle_\ast = -2l/7$\\
		$\langle Y,Y,Y,XY,XY^2 \rangle =2/49$\\
		$\langle Y,Y,Z,W,XY^2 \rangle_\ast =-2/1029$\\
		$\langle Y,Z,W,XY,Y^2\rangle_\ast =-1/1029$\\
		$\langle Z,Z,X,X^2,Y^2\rangle_\ast =-2a/7$\\
		$\langle Z,Z,Z,XY,XY\rangle_\ast =-2a/7$\\
		$\langle Z,Z,W,W,Y^2\rangle_\ast =-2/64827$\\
		$\langle W,W,W,X^2,Y^2\rangle_\ast =-2/(1361367k)$\\
		$\langle W,W,W,XY,XY\rangle_\ast =-2l/7$\\\\\hline
		6 pt correlators
	 	\\ \hline
		$\langle X,Y,Z,W,XY^2,XY^2 \rangle_\ast = 5/7203$\\
		$\langle Y,Z,Z,Z,Y^2,XY^2 \rangle_\ast = 6a/49$\\
		$\langle Y,W,W,W,Y^2,XY^2 \rangle_\ast = 6a/49$\\
		$\langle Z,Z,W,W,XY,XY^2 \rangle_\ast = 4/151263$\\ \\\hline
		7 pt correlators
	 	\\ \hline
		None\\
		\end{tabular}\\
\hline\hline
\end{tabular}

\begin{tabular}{||c|c||}
\hline\hline
\multicolumn{2}{||c||}{$W_{1,0}=x^4+y^6, \quad G =\langle  J = (1/4,1/6)\rangle$} \\\hline\hline
\multicolumn{2}{||c||}{$k \leq 7 $, $X = e_{9J}, Y =e_{2J}, Z = e_{7J}, W = xy^2 e_0$. $P = 7/12, \hat{c} = 7/6$. }\\
\multicolumn{2}{||c||}{Relations:  $XZ = Y^2 $, $X^2Z = 24W^2$, $X^3 = Y^3 = Z^2 = W^3 = 0$, $\dim = 9$.  }\\
\hline
 	\begin{tabular}{c}
 	{3 pt correlators}
 	\\ \hline
	\begin{tabular}{cc}
$\langle \one, \one, XY^2\rangle= 1$&
$\langle \one, X, Y^2\rangle= 1$\\ 
$\langle \one, Y, XY\rangle= 1$&
$\langle \one, Z, X^2\rangle= 1$\\ 
$\langle \one, W, W\rangle= 1/24$&
$\langle X, X, Z\rangle= 1$\\
$\langle X, Y, Y\rangle= 1$\\ 
\end{tabular}\\
5 pt correlators
	 	\\ \hline
		$\langle X,X,Z,XZ,X^2Z \rangle = 0$\\
		$\langle X,Y,Y,Y^2,XY^2 \rangle = 1/24$\\
		$\langle X,Y,Z,XY,XY^2 \rangle = 1/24$\\
		$\langle X,Z,Z,X^2,X^2Z \rangle = 0$\\
		$\langle X,Z,W,W,X^2Z \rangle_\ast = 1/576$\\
		$\langle X,W,W,XZ,XZ \rangle_\ast = 0$\\
		$\langle Y,Y,Y,XY,XY^2 \rangle =1/24$\\
		$\langle Y,Y,Z,X^2,X^2Z \rangle =1/24$\\
		$\langle Y,Y,W,W,XY^2 \rangle_\ast =-1/576$\\
		$\langle Y,W,W,XY,Y^2 \rangle_\ast =-1/576$\\
		$\langle Z,Z,Z,Z,X^2Z \rangle =1/8$\\
		$\langle Z,Z,Z,XZ,XZ \rangle =1/8$\\
		$\langle Z,W,W,X^2,XZ \rangle_\ast =1/576$\\
		$\langle Z,W,W,XY,XY\rangle_\ast =0$\\
		$\langle W,W,W,W,XZ \rangle_\ast =1/13824$\\\\\hline
7 pt correlators
	 	\\ 
		\hline
		$\langle Z,W,W,W,W,X^2Z,X^2Z \rangle_\ast = 1/55296$\\
		\hline
\end{tabular}&
\begin{tabular}{c}
4 pt correlators
	 	\\ \hline
		$\langle X,X,X,XY^2 \rangle = 1/6$\\
		$\langle X,X,X^2,Y^2 \rangle = 1/6$\\
		$\langle X,X,XY,XY \rangle = 1/3$\\
		$\langle X,Y,X^2,XY \rangle = 1/6$\\
		$\langle X,Z,Z,XZ \rangle = 0$\\
		$\langle X,Z,X^2,X^2 \rangle = 0$\\
		$\langle X,W,W,X^2 \rangle_\ast = -1/144$\\
		$\langle Y,Y,Z,Y^2 \rangle =1/4$\\
		$\langle Y,Y,X^2,X^2 \rangle =1/6$\\
		$\langle Y,Z,Z,XY \rangle =1/4$\\
		$\langle Z,Z,Z,X^2 \rangle =0$\\
		$\langle Z,Z,W,W \rangle_\ast =-1/96$\\\\\hline
6 pt correlators
	 	\\ \hline
		$\langle X,Z,Z,Z,X^2Z,X^2Z \rangle = 0$\\
		$\langle X,X,W,W,X^2Z,X^2Z \rangle_\ast = -1/1728$\\
		$\langle Y,Y,Z,Z,XY^2,XY^2 \rangle = 1/24$\\
		$\langle Z,Z,W,W,XZ,X^2Z \rangle_\ast = -1/1152$\\
		$\langle W,W,W,W,X^2,X^2Z \rangle_\ast = -1/27648$\\\\\\
\end{tabular}\\
\hline\hline
\end{tabular}

\begin{tabular}{||c|c||}
\hline\hline
\multicolumn{2}{||c||}{$Q_{2,0}=x^3+xy^4+yz^2, \quad G^{max} =\langle \gamma = (1/3,11/12,1/24)\rangle$} \\\hline\hline
\multicolumn{2}{||c||}{$k \leq 5 $, $ X = e_{7\gamma}, Y =e_{11\gamma}$. $P = 7/24, \hat{c} = 7/6$. }\\
\multicolumn{2}{||c||}{Relations:  $2X^7 = Y^3, X^8 = Y^4 = 0 $ , $\dim = 17$. }\\
\hline
 	\begin{tabular}{cc}
 	\multicolumn{2}{c}{3 pt correlators}
 	\\ \hline
$\langle \one, \one, X^7Y\rangle= -2$&
$\langle \one, X, X^7Y\rangle= -2$\\ 
$\langle \one, X^2, X^6Y\rangle= -2$&
$\langle \one, Y, X^7\rangle= -2$\\ 
$\langle \one, X^3, X^5Y\rangle= -2$&
$\langle \one, XY, X^6\rangle= -2$\\ 
$\langle \one, X^4, X^3Y\rangle= -2$&
$\langle \one, X^2Y, X^5\rangle= -2$\\ 
$\langle \one, Y^3, Y^3\rangle= -4$&
$\langle X, X, X^5Y\rangle= -2$\\ 
$\langle X, X^2, X^4Y\rangle= -2$&
$\langle X, Y, X^6\rangle= -2$\\ 
$\langle X, X^3, X^3Y\rangle_\ast= -2$& 
$\langle X, XY, X^5\rangle= -2$\\
$\langle X, X^4, X^2Y\rangle= -2$& 
$\langle X^2, Y, X^5\rangle= -2$\\
$\langle X^2, X^2, X^3Y\rangle_\ast= -2$& 
$\langle X^2, XY, X^4\rangle= -2$\\
$\langle X^2, X^3, X^2Y\rangle_\ast= -2$& 
$\langle Y, X^3, X^4\rangle= -2$\\ 
$\langle Y, Y, Y^2\rangle= -4$&
$\langle X^3, X^3, XY\rangle_\ast= -2$\\  
\end{tabular}&
\begin{tabular}{c}
4 pt correlators
	 	\\ \hline
		$\langle X,X,X^6,X^7Y \rangle = 2/3$\\
		$\langle X,X,X^7,X^6Y \rangle = 1/3$\\
		$\langle X,Y,X^7,X^7 \rangle = -1/3$\\
		$\langle X,Y,Y^2,Y^3 \rangle_\ast =-4/3$\\
		$\langle Y,Y,XY,X^7Y \rangle =-2/3$\\
		$\langle Y,Y,X^2Y,X^6Y \rangle =-2/3$\\
		$\langle Y,Y,X^3Y,X^5Y \rangle =-2/3$\\
		$\langle Y,Y,X^4Y,X^4Y \rangle_\ast =-2/3$\\
		\\\hline
		5 pt correlators
	 	\\ \hline
		None\\
\end{tabular}\\
\hline\hline
\end{tabular}

\begin{tabular}{||c|c||}
\hline\hline
\multicolumn{2}{||c||}{$Q_{2,0}=x^3+xy^4+yz^2, \quad G = \langle J = (1/3,1/6, 5/12) \rangle$} \\\hline\hline
\multicolumn{2}{||c||}{ $k \leq 7 $, $X = e_{10J}, Y =e_{8J}, Z = y^3e_{6J}, W = xy e_{6J}$. $P = 7/12, \hat{c} = 7/6$. }\\
\multicolumn{2}{||c||}{Relations:  $X^3Y = 8Z^2 = -24W^2, X^4 = Y^2 = Z^3 = W^3 = 0 $ $a = \pm1$, $\dim = 10$. }\\
\hline
 	\begin{tabular}{c}
 	{3 pt correlators}
 	\\ \hline
	\begin{tabular}{cc}
$\langle \one, \one, X^3Y\rangle= -2$&
$\langle \one, X, X^2Y\rangle= -2$\\ 
$\langle \one, Y, X^3\rangle= -2$&
$\langle \one, X^2, XY\rangle= 1$\\
$\langle \one, Z, Z\rangle= -1/4$& 
$\langle \one, W, W\rangle= 1/12$\\
$\langle X, X, XY\rangle_\ast= -2 $&
$\langle X, Y, X^2\rangle= 1$\\
\end{tabular}\\
		5 pt correlators
	 	\\ \hline
		$\langle X,X,Z,X^2Y, X^3Y \rangle_\ast =0$\\
		$\langle X,Y,Y,X^2Y, X^3Y \rangle =2/9$\\
		$\langle X,Y,Z,X^3, X^3Y \rangle_\ast =0$\\
		$\langle X,Z,Z,Z, X^3Y \rangle_\ast =a/72$\\
		$\langle X,Z,W,W, X^3Y \rangle_\ast =a/216$\\
		$\langle Y,Y,Y,X^3, X^3Y \rangle =2/9$\\
		$\langle Y,Y,Z,Z, X^3Y \rangle_\ast =-1/36$\\
		$\langle Y,Y,W,W, X^3Y \rangle_\ast =1/108$\\
		$\langle Y,Z,Z,XY, X^2Y \rangle_\ast =-1/36$\\
		$\langle Y,W,W,XY, X^2Y \rangle_\ast =-1/36$\\
		$\langle X^2,Z,Z,Z, X^2Y \rangle_\ast =a/72$\\
		$\langle X^2,Z,W,W, X^2Y \rangle_\ast =a/216$\\
		$\langle Z,Z,Z,XY,X^3 \rangle_\ast =a/72$\\
		$\langle Z,W,W,XY,X^3 \rangle_\ast =a/216$\\
7 pt correlators
	 	\\ \hline
		$\langle X^2, Z, Z, Z, Z, X^3Y, X^3Y \rangle_\ast =-1/216$\\
		$\langle X^2, Z, Z, W, W, X^3Y, X^3Y \rangle_\ast =1/1944$\\
		$\langle X^2, W, W, W, W, X^3Y, X^3Y \rangle_\ast =-1/1944$\\\\\\
\end{tabular}&
\begin{tabular}{c}
4 pt correlators
	 	\\ \hline
		$\langle X,X,X^2,X^3Y \rangle_\ast = 2/3$\\
		$\langle X,X,X^3,X^2Y \rangle_\ast = 4/3$\\
		$\langle X,Y,Z,X^2Y \rangle_\ast  = 0$\\
		$\langle X,Y,X^3,X^3 \rangle = 2/3$\\
		$\langle X,Z,Z,X^3 \rangle_\ast = -1/12$\\
		$\langle X,Z,XY,XY \rangle_\ast =0$\\
		$\langle X,W,W,X^3 \rangle_\ast = 1/36$\\
		$\langle Y,Y,Y,X^2Y \rangle =-2/3$\\
		$\langle Y,Y,Z,X^3 \rangle_\ast =0$\\
		$\langle Y,Y,XY,XY \rangle_\ast =-2/3$\\
		$\langle Y,X^2,Z,XY \rangle_\ast =0$\\
		$\langle Y,Z,Z,Z \rangle_\ast =a/48$\\
		$\langle Y,Z,W,W \rangle_\ast =a/144$\\
		$\langle X^2, X^2,Z,Z \rangle_\ast =1/12$\\
		$\langle X^2, X^2,W,W \rangle_\ast =1/36$\\
6 pt correlators
	 	\\ \hline
		$\langle X, Y , Z, Z,  X^3Y, X^3Y \rangle_\ast =-1/108$\\
		$\langle X, Y , W, W,  X^3Y, X^3Y \rangle_\ast =1/324$\\
		$\langle Y, Y , Y, Z,  X^3Y, X^3Y \rangle_\ast =0$\\
		$\langle Z, Z , Z, Z,  XY, X^3Y \rangle_\ast =-1/288$\\
		$\langle Z, Z , Z, Z,  X^2Y, X^2Y \rangle_\ast =-1/288$\\
		$\langle Z, Z , W, W,  XY, X^3Y \rangle_\ast =1/2592$\\
		$\langle Z, Z , W, W,  X^2Y, X^2Y \rangle_\ast =1/2592$\\
		$\langle W, W , W, W,  XY, X^3Y \rangle_\ast =-1/2592$\\
		$\langle W, W , W, W,  X^2Y, X^2Y \rangle_\ast =-1/2592$\\
\end{tabular}\\
\hline\hline
\end{tabular}

\begin{tabular}{||c|c||}
\hline\hline
\multicolumn{2}{||c||}{$Q_{2,0}^T=x^3y+y^4z+z^2, \quad G^{max} = \langle J = (7/24,1/8, 1/2)\rangle$} \\\hline\hline
\multicolumn{2}{||c||}{$k \leq 5 $, $ X = e_{11J}, Y =e_{9J}, Z = y^3e_{8J}$. $P = 5/12, \hat{c} = 7/6$. }\\
\multicolumn{2}{||c||}{Relations:  $X^4 = -3Y^2, aX^3Y=Z^2, X^8 = Y^3 = Z^3 =0 $, $a\neq 0$, $\dim = 14$. }\\
\hline
 	\begin{tabular}{c}
 	{3 pt correlators}
 	\\ \hline
	\begin{tabular}{cc}
$\langle \one, \one, X^3Y^2\rangle= 1$&
$\langle \one, X, X^2Y^2\rangle= 1$\\
$\langle \one, X^2, XY^2\rangle= 1$&
$\langle \one, Y, X^3Y\rangle= 1$\\
$\langle \one, X^3, Y^2\rangle= 1$&
$\langle \one, XY, X^2Y\rangle= a$\\
$\langle \one, Z, YZ\rangle= 1$&
$\langle X, X, XY^2\rangle= 1$\\ 
$\langle X, X^2, Y^2\rangle= 1$&
$\langle X, Y, X^2Y\rangle= 1$\\ 
$\langle X, XY, XY\rangle= 1$&
$\langle X, X^3, X^3\rangle_\ast= -3$\\
$\langle X^2, Y, XY\rangle= 1$&
$\langle X^2, X^2, X^3\rangle_\ast= -3$\\
$\langle Y, Z, Z\rangle_\ast= a$\\ 
\end{tabular}\\
5 pt correlators
	 	\\ \hline
		$\langle X,Y,Y,X^3Y^2,X^3Y^2 \rangle_\ast = -1/96$\\
		$\langle Y,Z,Z,XY^2,X^3Y^2 \rangle_\ast = -a/96$\\
		$\langle Y,Z,Z,X^2Y^2,X^2Y^2 \rangle_\ast = -a/96$\\

\end{tabular}&
\begin{tabular}{c}
4 pt correlators
	 	\\ \hline
		$\langle X,X,X^2Y,X^3Y^2 \rangle = 1/8$\\
		$\langle X,X,X^3Y,X^2Y^2 \rangle = 1/8$\\
		$\langle X,Y,Y^2,X^2Y^2 \rangle = -1/12$\\
		$\langle X,Y,X^3Y,X^3Y \rangle =0$\\
		$\langle X,Y,XY^2,XY^2 \rangle =-1/12$\\
		$\langle X,Y,X^3,X^3Y^2 \rangle_\ast =1/8$\\
		$\langle X,Z,Z,X^3Y^2 \rangle_\ast =-a/8$\\
		$\langle X,Z,YZ,X^3Y \rangle_\ast =-a/8$\\
		$\langle X^2,Z,Z,X^2Y^2 \rangle_\ast =-a/8$\\
		$\langle Y,Y,Y,X^3Y^2 \rangle =7/24$\\
		$\langle Y,Y,XY,X^2Y^2 \rangle =5/24$\\
		$\langle Y,Y,X^2Y,XY^2 \rangle =1/8$\\
		$\langle Y,Y,Y^2,X^3Y \rangle =1/24$\\
		$\langle Y,Y,YZ,YZ \rangle_\ast =a/4$\\
		$\langle Y,Z,Y^2,YZ \rangle_\ast =-a/24$\\
		$\langle Z,Z,X^3,XY^2 \rangle_\ast =-a/8$\\
		$\langle Z,Z,XY,X^3Y \rangle_\ast =-a/8$\\
		$\langle Z,Z,X^2Y,X^2Y \rangle_\ast =-a/8$\\
		$\langle Z,Z,Y^2,Y^2 \rangle_\ast =a/24$\\
		\\\hline
\end{tabular}\\
\hline\hline
\end{tabular}

\begin{tabular}{||c|c||}
\hline\hline
\multicolumn{2}{||c||}{$S_{1,0}=x^2y+y^2z+z^5, \quad 
G^{max}=\langle \gamma = (1/20, 9/10, 1/5)\rangle$} \\\hline\hline
\multicolumn{2}{||c||}{$k \leq 4 $, $X = e_{7\gamma}, Y =e_{\gamma}$. $P = 1/4, \hat{c} = 6/5$. }\\
\multicolumn{2}{||c||}{Relations:  $X^5 = Y^3 $, $X^9 = Y^4 = 0$, $\dim = 17$.  }\\
\hline
 	\begin{tabular}{cc}
 	\multicolumn{2}{c}{3 pt correlators}
 	\\ \hline
$\langle \one, \one, X^3Y^3\rangle= -2$&
$\langle \one, X, X^2Y^3\rangle= -2$\\ 
$\langle \one, Y, X^3Y^2\rangle= -2$&
$\langle \one, X^2, XY^3\rangle= -2$\\ 
$\langle \one, XY, X^2Y^2\rangle= -2$&
$\langle \one, X^3, Y^3\rangle= -2$\\ 
$\langle \one, Y^2, X^3Y\rangle= -2$&
$\langle \one, X^2Y, XY^2\rangle= -2$\\ 
$\langle \one, X^4,X^4\rangle= -2$&
$\langle X, X, XY^3\rangle= -2$\\ 
$\langle X, Y, X^2Y^2\rangle= -2$&
$\langle X, X^2, Y^3\rangle= -2$\\ 
$\langle X, XY, XY^2\rangle= -2$& 
$\langle X, Y^2 X^2Y\rangle= -2$\\ 
$\langle X, X^3, X^4\rangle= -2$& 
$\langle Y, Y, X^3Y\rangle_\ast= -2$\\ 
$\langle Y, X^2, XY^2\rangle= -2$&
$\langle Y, XY, X^2Y\rangle_\ast= -2$\\ 
$\langle Y, X^3, Y^2\rangle= -2$& 
$\langle X^2, X^2, X^4\rangle_\ast= -2$\\
$\langle X^2, XY, Y^2\rangle= -2$&
$\langle X^2, X^3, X^3\rangle_\ast= -2$\\
$\langle XY, XY, XY\rangle= -2$\\ 
\end{tabular}&
\begin{tabular}{c}
4 pt correlators
	 	\\ \hline
		$\langle X,X,X^3Y,X^3Y^3 \rangle = -2/5$\\
		$\langle X,X,X^3Y^2,X^3Y^2 \rangle = -2/5$\\
		$\langle X,Y,X^4,X^3Y^3 \rangle = -2/5$\\
		$\langle X,Y,Y^3,X^2Y^3 \rangle_\ast =-2/5$\\
		$\langle X,Y,XY^3,XY^3 \rangle_\ast =1/5$\\
		$\langle Y,Y,Y^2,X^3Y^3 \rangle_\ast =8/5$\\
		$\langle Y,Y,XY^2,X^2Y^3 \rangle_\ast =6/5$\\
		$\langle Y,Y,Y^3,X^3Y^2 \rangle_\ast =-2/5$\\
		$\langle Y,Y,X^2Y^2,XY^3 \rangle_\ast =4/5$\\
\end{tabular}\\
\hline\hline
\end{tabular}

\begin{tabular}{||c|c||}
\hline\hline
\multicolumn{2}{||c||}{$S_{1,0}=x^2y+y^2z+z^5, \quad G =\langle J =( 3/10 ,2/5 ,1/5)\rangle$} \\\hline\hline
\multicolumn{2}{||c||}{$k \leq 7 $, $X = e_{2J}, Y =e_{7J}, Z = e_{6J}, W = z^2e_{5J}$. $P = 3/5, \hat{c} = 6/5$. }\\
\multicolumn{2}{||c||}{Relations:  $2XZ = aY^2, X^3 = 2aYZ, XYZ = 10W^2, X^5 =Y^4 = Z^2 = 0  $, $a \neq 0$, $\dim = 10$.  }\\
\hline
 	\begin{tabular}{c}
 	{3 pt correlators}
 	\\ \hline
	\begin{tabular}{cc}
$\langle \one, \one, XYZ \rangle= 1$&
$\langle \one, X, YZ\rangle= 1$\\ 
$\langle \one, Y, XZ\rangle= 1$&
$\langle \one, Z, XY\rangle= 1$\\ 
$\langle \one, W, W\rangle= 1/10$&
$\langle \one, X^2, X^2\rangle= 2a$\\ 
$\langle X, Y,Z\rangle= 1$&
$\langle X, X,X^2\rangle_\ast= 2a$\\
$\langle Y,Y,Y\rangle_\ast= 2/a$\\
\end{tabular}\\
5 pt correlators
	 	\\ \hline
		$\langle X,X,Z,YZ,XYZ \rangle = 1/25$\\
		$\langle X,Y,Y, YZ,XYZ \rangle_\ast = 0$\\
		$\langle X,Y,Z,XZ,XYZ \rangle = -1/50$\\
		$\langle X,Z,Z,XY,XYZ \rangle = 0$\\
		$\langle X,Z,W,W,XYZ \rangle_\ast = -3/500$\\
		$\langle X,W,W,XZ,YZ \rangle_\ast = -1/500$\\
		$\langle Y,Y,Y,XZ,XYZ \rangle_\ast = -2/(25a)$\\
		$\langle Y,Y,Z,XY,XYZ \rangle_\ast = -2/(25a)$\\
		$\langle Y,Y,W,W,XYZ \rangle_\ast = 1/(125a)$\\
		$\langle Y,Z,Z,X^2,XYZ \rangle_\ast = 0$\\
		$\langle Y,Z,Z,YZ,YZ \rangle_\ast = 1/(25a)$\\
		$\langle Y,W,W,XY,YZ \rangle_\ast = 1/(250a)$\\
		$\langle Y,W,W,XZ,XZ \rangle_\ast = 1/250$\\
		$\langle Z,Z,Z,Z,XYZ \rangle = 3/25$\\
		$\langle Z,Z,Z,XZ,YZ \rangle = 2/25$\\
		$\langle Z,W,W,X^2,YZ \rangle_\ast = -1/125$\\
		$\langle Z,W,W,XY,XZ \rangle_\ast = -1/500$\\
		$\langle W,W,W,W,XZ \rangle_\ast = -1/5000$\\
7 pt correlators
	 	\\ \hline
		$\langle Y,W,W,W,W,XYZ,XYZ \rangle =-3/(62500a)$\\

\end{tabular}&
\begin{tabular}{c}
4 pt correlators
	 	\\ \hline
		$\langle X,X,Y,XY^2 \rangle = 1/5$\\
		$\langle X,X,XY,YZ \rangle = 1/5$\\
		$\langle X,X,XZ,XZ \rangle_\ast = a/5$\\
		$\langle X,Y,X^2,YZ \rangle_\ast =2/5$\\
		$\langle X,Y,XY,XZ \rangle =1/5$\\
		$\langle X,Z,Z,YZ \rangle =1/10$\\
		$\langle X,Z,X^2,XZ \rangle_\ast =a/5$\\
		$\langle X,Z,XY,XY \rangle =0$\\
		$\langle X,W,W,XY \rangle =-1/50$\\
		$\langle Y,Y,Z,YZ \rangle_\ast =-1/(5a)$\\
		$\langle Y,Y,X^2,XZ \rangle_\ast =0$\\
		$\langle Y,Y,XY,XY \rangle_\ast =2/(5a)$\\
		$\langle Y,Z,Z,XZ \rangle =-1/5$\\
		$\langle Y,Z,X^2,XY \rangle_\ast =0$\\
		$\langle Y,W,W,X^2 \rangle_\ast =-1/50$\\
		$\langle Z,Z,Z,XY \rangle =-1/10$\\
		$\langle Z,Z,W,W \rangle_\ast =1/50$\\
		$\langle Z,Z,X^2,X^2 \rangle_\ast =2/5$\\
6 pt correlators
	 	\\ \hline
		$\langle X,Z,Z,Z,XYZ,XYZ \rangle = 0$\\
		$\langle X,X,Z, Z,XYZ, XYZ \rangle_\ast = 1/625$\\
		$\langle Y,Y,Z,Z, XYZ,XYZ \rangle_\ast = 4/(125a)$\\
		$\langle Y,Z,W,W, YZ,XYZ \rangle_\ast = 1/(2500a)$\\
		$\langle Z,Z,W, W,XZ, XYZ \rangle_\ast = -3/2500$\\
		$\langle W,W,W,W,X^2,XYZ \rangle_\ast =-9/25000$\\
		$\langle W,W,W,W,YZ,YZ \rangle_\ast =3/(12500a)$\\

\end{tabular}\\
\hline\hline
\end{tabular}

\begin{tabular}{||c|c||}
\hline\hline
\multicolumn{2}{||c||}{$E_{19}=x^3+xy^7, \quad G^{max} = \langle J = (1/3,2/21)\rangle$} \\\hline\hline
\multicolumn{2}{||c||}{$k \leq 5 $, $X = e_{13J}, Y =e_{11J}$. $P = 2/7, \hat{c} = 8/7$. }\\
\multicolumn{2}{||c||}{Relations:  $-7X^6 = Y^3, X^7 = Y^5 = 0 $, $\dim = 15$.  }\\
\hline
 	\begin{tabular}{cc}
 	\multicolumn{2}{c}{3 pt correlators}
 	\\ \hline
$\langle \one, \one, X^6Y\rangle= 1$&
$\langle \one, X, X^5Y\rangle= 1$\\ 
$\langle \one, Y, X^6\rangle= 1$&
$\langle \one, X^2, X^4Y\rangle= 1$\\
$\langle \one, XY, X^5\rangle= 1$&
$\langle \one, X^3, X^3Y\rangle= 1$\\ 
$\langle \one, X^2Y, X^4\rangle= 1$&
$\langle \one, Y^2, Y^2\rangle= -7$\\ 
$\langle X, X, X^4Y\rangle= 1$&
$\langle X, Y, X^5\rangle= 1$\\ 
$\langle X, X^2, X^3Y\rangle= 1$&
$\langle X, XY, X^4\rangle= 1$\\
$\langle X, X^3, X^2Y\rangle= 1$&
$\langle X^2, X^2, Y^2\rangle= 1$\\ 
$\langle X^2, Y, X^2Y\rangle= 1$&
$\langle X^2, XY, XY\rangle= 1$\\ 
$\langle X^3, Y, XY\rangle= 1$&
$\langle Y, X^2, X^4\rangle= 1$\\ 
$\langle Y, X^3, X^3\rangle= 1$& 
$\langle Y, Y, Y^2\rangle= -7$\\ 
$\langle X^2, X^2, X^2Y\rangle= 1$&
$\langle X^2, XY, X^3\rangle= 1$\\
\end{tabular}&
\begin{tabular}{c}
4 pt correlators
	 	\\ \hline
		$\langle X,X,X^5,X^6Y \rangle = 2/21$\\
		$\langle X,X,X^6,X^5Y \rangle = 1/21$\\
		$\langle X,Y,X^6,X^6 \rangle = -1/21$\\
		$\langle X,Y,Y^2,X^6Y \rangle_\ast = 1/3$\\
		$\langle Y,Y,XY,X^6Y \rangle =1/3$\\
		$\langle Y,Y,X^2Y,X^5Y \rangle =1/3$\\
		$\langle Y,Y,X^3Y,X^4Y \rangle =1/3$\\\\\hline
		5 pt correlators
	 	\\ \hline
		None\\
\end{tabular}\\
\hline\hline
\end{tabular}

\begin{tabular}{||c|c||}
\hline\hline
\multicolumn{2}{||c||}{$Z_{17}^T=x^3+xy^8, \quad G^{max} = \langle \gamma = (2/3,1/24)\rangle$} \\\hline\hline
\multicolumn{2}{||c||}{$k \leq 5$, $X = e_{5\gamma}, Y =e_{\gamma}$. $P = 7/24, \hat{c} = 7/6$. }\\
\multicolumn{2}{||c||}{Relations:  $-8X^7 = Y^3, X^8=XY^2=Y^4=0$, $\dim = 17$.  }\\
\hline
 	\begin{tabular}{cc}
 	\multicolumn{2}{c}{3 pt correlators}
 	\\ \hline
$\langle \one, \one, X^7Y\rangle= 1$&
$\langle \one, X, X^6Y\rangle= 1$\\ 
$\langle \one, X^2, X^5Y\rangle= 1$&
$\langle \one, Y, X^7\rangle= 1$\\ 
$\langle \one, X^3, X^4Y\rangle= 1$&
$\langle \one, XY, X^6\rangle= 1$\\
$\langle \one, X^4, X^3Y\rangle= 1$& 
$\langle \one, X^2Y, X^5\rangle= 1$\\
$\langle \one, Y^2, Y^2\rangle= -8$&
$\langle X, X, X^5Y\rangle= 1$\\
$\langle X, X^2, X^4Y\rangle= 1$& 
$\langle X, Y, X^6\rangle= 1$\\
$\langle X, X^3, X^3Y\rangle= 1$& 
$\langle X, XY, X^5\rangle= 1$\\
$\langle X, X^4, X^2Y\rangle= 1$& 
$\langle X^2, X^2, X^3Y\rangle= 1$\\ 
$\langle X^2, Y, X^5\rangle= 1$&
$\langle X^2, X^3, X^2Y\rangle= 1$\\ 
$\langle X^2, XY, X^4\rangle= 1$& 
$\langle X^3, Y, XY\rangle= 1$\\
$\langle Y, X^3, X^4\rangle= 1$& 
$\langle Y, Y, Y^2\rangle= \pm1$\\ 
$\langle X^3, X^3, XY\rangle= 1$\\  
\end{tabular}&
\begin{tabular}{c}
4 pt correlators
	 	\\ \hline
		$\langle X,X,X^6,X^7Y \rangle = 1/12$\\
		$\langle X,X,X^7,X^6Y \rangle = 1/24$\\
		$\langle X,Y,X^7,X^7 \rangle = -1/24$\\
		$\langle X,Y,Y^2,X^7Y \rangle_\ast =1/3$\\
		$\langle Y,Y,XY,X^7Y \rangle =1/3$\\
		$\langle Y,Y,X^2Y,X^6Y \rangle =1/3$\\
		$\langle Y,Y,X^3Y,X^5Y \rangle =1/3$\\
		$\langle Y,Y,X^4Y,X^4Y \rangle =1/3$\\
		\\\hline
		5 pt correlators
	 	\\ \hline
		None\\
\end{tabular}\\
\hline\hline
\end{tabular}

\begin{tabular}{||c|c||}
\hline\hline
\multicolumn{2}{||c||}{$Z_{17}^T=x^3+xy^8, \quad G=\langle J = (1/3,1/12)\rangle$} \\\hline\hline
\multicolumn{2}{||c||}{$k \leq7 $, $X = e_{4J}, Y =e_{2J}, Z = y^7e_0, W = xy^3e_0$. $P = 7/12, \hat{c} = 7/6$. }\\
\multicolumn{2}{||c||}{Relations:  $X^3Y = -8Z^2 = 24W^2, X^4 = Y^2 Z^3 = W^3 $, $a,b \neq 0$ , $\dim = 10$. }\\
\hline
 	\begin{tabular}{c}
{3 pt correlators}	\\ \hline
\begin{tabular}{cc}
$\langle \one, \one, X^3Y^\rangle= 1$&
$\langle \one, X, X^2Y\rangle= 1$\\ 
$\langle \one, Y, X^3\rangle= 1$&
$\langle \one, X^2, XY\rangle= 1$\\
$\langle \one, Z, Z\rangle= -1/8$&
$\langle \one, W, W\rangle= 1/24$\\ 
$\langle X, X, XY\rangle= 1$&
$\langle X, Y, X^2\rangle= 1$\\ 
\end{tabular}\\
4 pt correlators
	 	\\ \hline
		$\langle X,X,X^2,X^3Y \rangle = 1/12$\\
		$\langle X,X,X^3,X^2Y \rangle = 1/6$\\
		$\langle X,Y,X^3,X^3 \rangle = 1/12$\\
		$\langle X,Y,Z,X^2Y \rangle_\ast =0$\\
		$\langle X,Z,Z,X^3 \rangle_\ast =1/96$\\
		$\langle X,Z,XY,XY \rangle_\ast =0$\\
		$\langle X,W,W,X^3 \rangle_\ast =-1/288$\\
		$\langle Y,Y,Y,X^2Y \rangle =1/3$\\
		$\langle Y,Y,XY,XY \rangle =1/3$\\
		$\langle Y,Y,Z,X^3 \rangle_\ast =0$\\
		$\langle Y,X^2,Z,XY \rangle_\ast =0$\\
		$\langle Y,Z,Z,Z \rangle_\ast =a = \pm1/192$\\
		$\langle Y,Z,W,W \rangle_\ast =b = \pm 1/576$\\
		$\langle X^2,X^2,Z,Z  \rangle_\ast =1/96$\\
		$\langle X^2,X^2,W,W \rangle_\ast =-1/288$\\\\\hline
7 pt correlators
	 	\\ \hline
		None\\

\end{tabular}&
\begin{tabular}{c}
5 pt correlators
	 	\\ \hline
		$\langle X,Y,Y, X^2Y,X^3Y \rangle = 1/36$\\
		$\langle X,X,Z,X^2Y,X^3Y \rangle_\ast = 0$\\
		$\langle X,Y,Z,X^3,X^3Y \rangle_\ast = 0$\\
		$\langle X,Z,Z,Z,X^3Y \rangle_\ast =-a/6$\\
		$\langle X,Z,W,W,X^3Y \rangle_\ast =-b/6$\\
		$\langle Y,Y,Y,X^3,X^3Y \rangle =1/36$\\
		$\langle Y,Y,Z,Z,X^3Y \rangle_\ast =1/288$\\
		$\langle Y,Y,W,W,X^3Y \rangle_\ast =-1/864$\\
		$\langle Y,Z,Z,XY,X^2Y \rangle_\ast =1/288$\\
		$\langle Y,W,W,XY,X^2Y \rangle_\ast =-1/864$\\
		$\langle X^2,Z,Z,Z,X^2Y  \rangle_\ast =-a/6$\\
		$\langle X^2,Z,W,W,X^2Y \rangle_\ast =-b/6$\\
		$\langle Z,Z,Z,XY,X^3 \rangle_\ast =-a/6$\\
		$\langle Z,W,W,XY,X^3 \rangle_\ast =-b/6$\\
		\\\hline
6 pt correlators
	 	\\ \hline
		$\langle X,Y,Z,Z,X^3Y,X^3Y \rangle_\ast = -1/3456 $\\
		$\langle X,Y,W,W,X^3Y,X^3Y\rangle_\ast = 1/10368$\\
		$\langle Y,Y,Y,Z,X^3Y,X^3Y \rangle_\ast = 0$\\
		$\langle Z,Z,Z,Z,XY,X^3Y \rangle_\ast =1/9216$\\
		$\langle Z,Z,Z,Z,X^2Y,X^2Y \rangle_\ast =1/9216$\\
		$\langle Z,Z,W,W,XY,X^3Y \rangle_\ast =-1/82944$\\
		$\langle Z,Z,W,W,X^2Y,X^2Y \rangle_\ast =-1/82944$\\
		$\langle W,W,W,W,XY,X^3Y \rangle =1/82944$\\
		$\langle W,W,W,W,X^2Y,X^2Y \rangle =1/82944$\\
		\\\hline
\end{tabular}\\
\hline\hline
\end{tabular}

\begin{tabular}{||c|c||}
\hline\hline
\multicolumn{2}{||c||}{$Z_{19}=x^3y+y^9, \quad G^{max}= \langle J = (8/27,1/9)\rangle$} \\\hline\hline
\multicolumn{2}{||c||}{$k \leq 5 $, $X = e_{11J}, Y =e_{19J}$. $P = 1/3, \hat{c} = 32/27$. }\\
\multicolumn{2}{||c||}{Relations:  $X^9 = -3Y^2, X^{16} = Y^3 = 0 $, $\dim = 25$.  }\\
\hline
 	\begin{tabular}{c}
 	{3 pt correlators}
 	\\ \hline
	\begin{tabular}{cc}
$\langle \one, \one, X^7Y^2\rangle= 1$&
$\langle \one, X, X^6Y^2\rangle= 1$\\ 
$\langle \one, X^2, X^5Y^2\rangle= 1$&
$\langle \one, X^3, X^4Y^2\rangle= 1$\\ 
$\langle \one, X^4, X^3Y^2\rangle= 1$&
$\langle \one, Y, X^7Y^2\rangle= 1$\\
$\langle \one, X^5, X^2Y^2\rangle= 1$& 
$\langle \one, XY, X^6Y\rangle= 1$\\
$\langle \one, X^6, XY^2\rangle= 1$& 
$\langle \one, X^2Y, X^5Y\rangle= 1$\\ 
$\langle \one, X^7, Y^2\rangle= 1$&
$\langle \one, X^3Y, X^4Y\rangle= 1$\\
$\langle \one, X^8, X^8\rangle= -3$&
$\langle X, X, X^5Y^2\rangle= 1$\\
$\langle X, X^2, X^4Y^2\rangle= 1$&
$\langle X, X^3, X^3Y^2\rangle= 1$\\ 
$\langle X, X^4, X^2Y^2\rangle= 1$&
$\langle X, Y, X^6Y\rangle= 1$\\ 
$\langle X, X^5, XY^2\rangle= 1$& 
$\langle X, XY, X^5Y\rangle= 1$\\ 
$\langle X, X^6, Y^2\rangle= 1$& 
$\langle X, X^2Y, X^4Y\rangle= 1$\\ 
$\langle X, X^7, X^8\rangle= -3$& 
$\langle X, X^3Y, X^4Y\rangle= 1$\\ 
$\langle X^2, X^2, X^3Y^2\rangle= 1$&
$\langle X^2, X^3, X^2Y^2\rangle= 1$\\
$\langle X^2, X^4, XY^2\rangle= 1$&
$\langle X^2, Y, X^5Y\rangle= 1$\\ 
$\langle X^2, X^5, Y^2\rangle= 1$&
$\langle X^2, XY, X^4Y\rangle= 1$\\ 
$\langle X^2, X^6, X^8\rangle= -3$& 
$\langle X^2, X^2Y, X^3Y\rangle= 1$\\ 
$\langle X^2, X^7, X^7\rangle_\ast= -3$&
$\langle X^3, X^3, XY^2\rangle= 1$\\ 
$\langle X^3, X^4, Y^2\rangle= 1$&
$\langle X^3, Y, X^4Y\rangle= 1$\\ 
$\langle X^3, X^5, X^8\rangle= -3$& 
$\langle X^3, XY, X^3Y\rangle= 1$\\ 
$\langle X^3, X^6, X^7\rangle_\ast= -3$& 
$\langle X^3, X^2Y, X^2Y\rangle= 1$\\ 
$\langle X^4, X^4, X^8\rangle= -3$& 
$\langle X^4, Y, X^3Y\rangle= 1$\\ 
$\langle X^4, X^5, X^7\rangle_\ast= -3$& 
$\langle X^4, XY, X^2Y\rangle= 1$\\ 
$\langle X^4, X^6, X^6\rangle_\ast= -3$& 
$\langle Y, Y, X^7\rangle= 1$\\
$\langle Y, X^5, X^2Y\rangle= 1$&
$\langle Y, XY, X^6\rangle= 1$\\
$\langle X^5, X^5, X^6\rangle_\ast= -3$&
$\langle X^5, XY, XY\rangle= 1$\\  
\end{tabular}\\\\
		5 pt correlators\\\hline
		$\langle X^4,X^4,X^4,X^6Y^2,X^7Y^2 \rangle = 4/81$\\
		$\langle X^3, X^4,X^4,X^7Y^2,X^7Y^2 \rangle_\ast = 2/81$\\\\\\\\\\\\\\\\\\\\\
\end{tabular}&
\begin{tabular}{c}
4 pt correlators
	 	\\ \hline
		$\langle X,X,X^7Y,X^7Y^2 \rangle = 1/9$\\
		$\langle X,X^2,X^6Y,X^7Y^2 \rangle = 1/9$\\
		$\langle X,X^2,X^7Y,X^6Y^2 \rangle = 1/9$\\
		$\langle X,X^3,X^5Y,X^7Y^2 \rangle = 1/9$\\
		$\langle X,X^3,X^6Y,X^6Y^2 \rangle = 1/9$\\
		$\langle X,X^3,X^7Y,X^5Y^2 \rangle = 1/9$\\
		$\langle X,X^4,X^4Y,X^7Y^2 \rangle = 1/9$\\
		$\langle X,X^4,X^5Y,X^6Y^2 \rangle = 1/9$\\
		$\langle X,X^4,X^6Y,X^5Y^2 \rangle = 1/9$\\
		$\langle X,X^4,X^7Y,X^4Y^2 \rangle = 1/9$\\
		$\langle X^2,X^2,X^5Y,X^7Y^2 \rangle = 1/9$\\
		$\langle X^2,X^2,X^6Y,X^6Y^2 \rangle = 2/9$\\
		$\langle X^2,X^2,X^7Y,X^5Y^2 \rangle = 1/9$\\
		$\langle X^2,X^3,X^4Y,X^7Y^2 \rangle = 1/9$\\
		$\langle X^2,X^3,X^5Y,X^6Y^2 \rangle = 2/9$\\
		$\langle X^2,X^3,X^6Y,X^5Y^2 \rangle = 2/9$\\
		$\langle X^2,X^3,X^7Y,X^4Y^2 \rangle = 1/9$\\
		$\langle X^2,X^4,X^3Y,X^7Y^2 \rangle = 1/9$\\
		$\langle X^2,X^4,X^4Y,X^6Y^2 \rangle = 2/9$\\
		$\langle X^2,X^4,X^5Y,X^5Y^2 \rangle = 2/9$\\
		$\langle X^2,X^4,X^6Y,X^4Y^2 \rangle = 2/9$\\
		$\langle X^2,X^4,X^7Y,X^3Y^2 \rangle = 1/9$\\
		$\langle X^3,X^3,X^3Y,X^7Y^2 \rangle = 1/9$\\
		$\langle X^3,X^3,X^4Y,X^6Y^2 \rangle = 2/9$\\
		$\langle X^3,X^3,X^5Y,X^5Y^2 \rangle = 1/3$\\
		$\langle X^3,X^3,X^6Y,X^4Y^2 \rangle = 2/9$\\
		$\langle X^3,X^3,X^7Y,X^3Y^2 \rangle = 1/9$\\
		$\langle X^3,X^4,X^2Y,X^7Y^2 \rangle = 1/9$\\
		$\langle X^3,X^4,X^3Y,X^6Y^2 \rangle = 2/9$\\
		$\langle X^3,X^4,X^4Y,X^5Y^2 \rangle = 1/3$\\
		$\langle X^3,X^4,X^5Y,X^4Y^2 \rangle = 1/3$\\
		$\langle X^3,X^4,X^6Y,X^3Y^2 \rangle = 2/9$\\
		$\langle X^3,X^4,X^7Y,X^2Y^2 \rangle = 1/9$\\
		$\langle X^4,X^4,XY,X^7Y^2 \rangle = 1/9$\\
		$\langle X^4,X^4,X^2Y,X^6Y^2 \rangle = 2/9$\\
		$\langle X^4,X^4,X^3Y,X^5Y^2 \rangle = 1/3$\\
		$\langle X^4,X^4,X^4Y,X^4Y^2 \rangle = 4/9$\\
		$\langle X^4,X^4,X^5Y,X^3Y^2 \rangle = 1/3$\\
		$\langle X^4,X^4,X^6Y,X^2Y^2 \rangle = 2/9$\\
		$\langle X^4,X^4,X^7Y,XY^2 \rangle = 1/9$\\
 \end{tabular}\\
\hline\hline
\end{tabular}

\begin{tabular}{||c|c||}
\hline\hline
\multicolumn{2}{||c||}{$Z_{19}^T=x^3+xy^9, \quad G^{max} = \langle J = (1/3,2/27)\rangle$} \\\hline\hline
\multicolumn{2}{||c||}{$k \leq 5 $, $X = e_{16J}, Y =e_{14J}$. $P = 8/27, \hat{c} = 32/27$. }\\
\multicolumn{2}{||c||}{Relations:  $-9X^8 = Y^3, X^9 = Y^5 =0$, $\dim =19$.  }\\
\hline
 	\begin{tabular}{cc}
 	\multicolumn{2}{c}{3 pt correlators}
 	\\ \hline
$\langle \one, \one, X^8Y\rangle= 1$&
$\langle \one, X, X^7Y\rangle= 1$\\ 
$\langle \one, X^2, X^6Y\rangle= 1$&
$\langle \one, Y, X^8\rangle= 1$\\ 
$\langle \one, X^3, X^5Y\rangle= 1$&
$\langle \one, XY, X^7\rangle= 1$\\ 
$\langle \one, X^4, X^4Y\rangle= 1$&
$\langle \one, X^2Y, X^6\rangle= 1$\\ 
$\langle \one, X^5, X^3Y\rangle= 1$&
$\langle \one, Y^2, Y^2\rangle= -9$\\ 
$\langle X, X, X^6Y\rangle= 1$&
$\langle X, X^2, X^5Y\rangle= 1$\\ 
$\langle X, Y, X^7\rangle= 1$&
$\langle X, X^2, X^5Y\rangle= 1$\\ 
$\langle X, XY, X^6Y\rangle= 1$& 
$\langle X, X^3, X^4Y\rangle= 1$\\
$\langle X, X^2Y, X^5\rangle= 1$& 
$\langle X, X^4, X^3Y\rangle= 1$\\ 
$\langle X^2, X^2, X^4Y\rangle= 1$& 
$\langle X^2, Y, X^6\rangle= 1$\\
$\langle X^2, X^3, X^3Y\rangle= 1$&
$\langle X^2, XY, X^5\rangle= 1$\\ 
$\langle X^2, X^4, X^2Y\rangle= 1$& 
$\langle Y, Y, Y^2\rangle= -9$\\ 
$\langle Y, X^3, X^5\rangle= 1$&
$\langle Y, X^4, X^4\rangle= 1$\\ 
$\langle X^3, X^3, X^2Y\rangle= 1$&
$\langle X^3, XY, X^4\rangle= 1$\\
\end{tabular}&
\begin{tabular}{c}
4 pt correlators
	 	\\ \hline
		$\langle X,X,X^7,X^8Y \rangle = 2/27$\\
		$\langle X,X,X^8,X^7Y \rangle = 1/27$\\
		$\langle X,Y,X^8,X^8 \rangle = -1/27$\\
		$\langle X,Y,Y^2,X^8Y \rangle_\ast = 1/3$\\
		$\langle Y,Y,XY,X^8Y \rangle =1/3$\\
		$\langle Y,Y,X^2Y,X^7Y \rangle =1/3$\\
		$\langle Y,Y,X^3Y,X^6Y \rangle =1/3$\\
		$\langle Y,Y,X^4Y,X^5Y \rangle =1/3$\\
		5 pt correlators
	 	\\ \hline
		None\\
\end{tabular}\\
\hline\hline
\end{tabular}

\begin{tabular}{||c|c||}
\hline\hline
\multicolumn{2}{||c||}{$W_{17}=x^4+Xy^5, \quad G^{max} = \langle J = (1/4,3/20)\rangle$} \\\hline\hline
\multicolumn{2}{||c||}{$k \leq 4 $, $X = e_{14J}, Y =e_{9J}$. $P = 1/5, \hat{c} = 6/5$. }\\
\multicolumn{2}{||c||}{Relations:  $X^4 = -5Y^4, X^7 = Y^5 = 0 $, $\dim = 16$.  }\\
\hline
 	\begin{tabular}{c}
 	{3 pt correlators}
 	\\ \hline
	\begin{tabular}{cc}
$\langle \one, \one, X^2Y^4\rangle= 1$&
$\langle \one, X, XY^4\rangle= 1$\\ 
$\langle \one, Y, X^2Y^3\rangle= 1$&
$\langle \one, X^2, Y^4\rangle= 1$\\ 
$\langle \one, XY, XY^3\rangle= 1$&
$\langle \one, Y^2, X^2Y^2\rangle= 1$\\ 
$\langle \one, X^2Y, Y^3\rangle= 1$&
$\langle \one, XY^2, XY^2\rangle= 1$\\ 
$\langle \one, X^3,X^3\rangle= -5$&
$\langle X, X, Y^4\rangle= 1$\\
$\langle X, Y, XY^3\rangle= 1$&
$\langle X, X^2, X^3\rangle= -5$\\ 
$\langle X, XY, Y^3\rangle= 1$& 
$\langle X, Y^2, XY^2\rangle= 1$\\ 
$\langle Y, Y, X^2Y^2\rangle= 1$&
$\langle Y, X^2, Y^3\rangle= 1$\\
$\langle Y, XY, XY^2\rangle= 1$&
$\langle Y, Y^2, X^2Y\rangle= 1$\\
$\langle X^2, Y^2, Y\rangle= 1$&
$\langle X^2, X^2, X^2\rangle_\ast= -5$\\
$\langle XY, XY, Y^2\rangle= 1$\\
\hline
\end{tabular}\\
\end{tabular}&
\begin{tabular}{c}
4 pt correlators
	 	\\ \hline
		$\langle X,X,X^2Y,X^2Y^4 \rangle = 1/4$\\
		$\langle X,X,X^2Y^2,X^2Y^3 \rangle = 1/4$\\
		$\langle X,Y,X^3,X^2Y^4 \rangle_\ast = 1/4$\\
		$\langle X,Y,Y^4,XY^4 \rangle= -1/20$\\
		$\langle Y,Y,Y^3,X^2Y^4 \rangle =3/20$\\
		$\langle Y,Y,XY^3,XY^4 \rangle =1/10$\\
		$\langle Y,Y,Y^4,X^2Y^3 \rangle =1/20$
\end{tabular}\\
\hline\hline
\end{tabular}

\begin{tabular}{||c|c||}
\hline\hline
\multicolumn{2}{||c||}{$W_{17}^T=x^4y+y^5, \quad G^{max} = \langle \gamma = (1/20,4/5)\rangle$} \\\hline\hline
\multicolumn{2}{||c||}{$k \leq 4$, $X = e_{3\gamma}, Y =e_{9\gamma}$. $P = 1/4, \hat{c} = 6/5$. }\\
\multicolumn{2}{||c||}{Relations:  $X^5 = -4Y^3, X^8 = Y^4 = 0 $, $\dim = 17$.  }\\
\hline
 	\begin{tabular}{c}
 	{3 pt correlators}
 	\\ \hline
	\begin{tabular}{cc}
$\langle \one, \one, X^3Y^3 \rangle= 1$&
$\langle \one, X, X^2Y^3\rangle= 1$\\ 
$\langle \one, Y, X^3Y^2\rangle= 1$&
$\langle \one, X^2, XY^3\rangle= 1$\\ 
$\langle \one, XY, X^2Y^2\rangle= 1$&
$\langle \one, X^3, Y^3\rangle= 1$\\ 
$\langle \one, Y^2, X^3Y\rangle= 1$&
$\langle \one, X^2Y, XY^2\rangle= 1$\\ 
$\langle \one, X^4, X^4\rangle= -4$& 
$\langle X, X, XY^3\rangle= 1$\\
$\langle X, Y, X^2Y^2\rangle= 1$& 
$\langle X, X^2, Y^3\rangle= 1$\\
$\langle X, XY, XY^2\rangle= 1$& 
$\langle X, X^3, X^4\rangle= -4$\\ 
$\langle X, Y^2, X^2Y\rangle= 1$& 
$\langle Y, Y, X^3Y\rangle= 1$\\
$\langle Y, X^2, XY^2\rangle= 1$&
$\langle Y, XY, X^2Y\rangle= 1$\\
$\langle Y, X^3, Y^2\rangle= 1$&
$\langle X^2, XY, Y^2\rangle= 1$\\
$\langle X^2, X^2, X^4\rangle= -4$&
$\langle X^2, X^3, X^3\rangle_\ast= -4$\\
$\langle XY, XY, XY\rangle= 1$\\
\end{tabular}\\
\end{tabular}&
\begin{tabular}{c}
4 pt correlators
	 	\\ \hline
		$\langle X,X,X^3Y,X^3Y^3 \rangle = 1/5$\\
		$\langle X,X,X^3Y^2,X^3Y^2 \rangle = 1/5$\\
		$\langle X,Y,X^4,X^3Y^3 \rangle_\ast = 1/5$\\
		$\langle X,Y,Y^3,X^2Y^3 \rangle= 1/20$\\
		$\langle X,Y,XY^3,XY^3 \rangle = -1/20$\\
		$\langle Y,Y,Y^2,X^3Y^3 \rangle =1/5$\\
		$\langle Y,Y,XY^2,X^2Y^3 \rangle =3/20$\\
		$\langle Y,Y,Y^3,X^3Y^2 \rangle =1/20$\\
		$\langle Y,Y,X^2Y^2,XY^3 \rangle =1/20$\\
\end{tabular}\\
\hline\hline
\end{tabular}

\begin{tabular}{||c|c||}
\hline\hline
\multicolumn{2}{||c||}{$W_{17}^T=x^4y+y^5, \quad G= \langle \gamma = (1/10,3/5)\rangle$} \\\hline\hline
\multicolumn{2}{||c||}{$k \leq 7 $, $X = e_{\gamma}, Y =e_{4\gamma}, Z = e_{7\gamma}, W = xy^2e_0$. $P = 3/5, \hat{c} = 6/5$. }\\
\multicolumn{2}{||c||}{Relations:  $XZ = Y^2, X^3 = -4YZ, XYZ = 20W^2, X^5 = Y^5 =Z^2 = W^3 =0 $, $\dim = 10$.  }\\
\hline
 	\begin{tabular}{c}
 	{3 pt correlators}
 	\\ \hline
	\begin{tabular}{cc}
$\langle \one, \one, XYZ\rangle= 1$&
$\langle \one, X, YZ\rangle= 1$\\ 
$\langle \one, Y, XZ\rangle= 1$&
$\langle \one, Z, XY\rangle= 1$\\ 
$\langle \one, W, W\rangle= 1/20$&
$\langle \one, X^2, X^2\rangle= -4$\\ 
$\langle X, X, X^2\rangle= -4$&
$\langle X, Y, Z\rangle= 1$\\ 
$\langle Y, Y, Y\rangle= 1$\\\hline
\end{tabular}\\
		5 pt correlators
	 	\\ \hline
		$\langle X,X,Z,YZ,XYZ \rangle = -1/25$\\
		$\langle X,Y,Y,YZ,XYZ \rangle = 0$\\
		$\langle X,Y,Z,XZ,XYZ \rangle = 1/50$\\
		$\langle X,Z,Z,XY,XYZ \rangle= 0$\\
		$\langle X,Z,W,W,XYZ \rangle_\ast = 3/1000$\\
		$\langle X,W,W,XZ,YZ \rangle_\ast =1/1000$\\
		$\langle Y,Y,Y,XZ,XYZ \rangle =1/25$\\
		$\langle Y,Y,Z,XY,XYZ \rangle =1/25$\\
		$\langle Y,Y,W,W,XYZ \rangle_\ast =-1/500$\\
		$\langle Y,Z,Z,YZ,YZ \rangle =1/50$\\
		$\langle Y,Z,Z,X^2,XYZ \rangle_\ast =0$\\
		$\langle Y,W,W,XY,YZ\rangle_\ast =-1/1000$\\
		$\langle Y,W,W,XZ,XZ \rangle_\ast =-1/500$\\
		$\langle Z,Z,Z,Z,XYZ \rangle =3/25$\\
		$\langle Z,Z,Z,XZ,YZ \rangle =2/25$\\
		$\langle Z,W,W,X^2,XZ \rangle_\ast =1/250$\\
		$\langle Z,W,W,XY,XZ \rangle_\ast =1/1000$\\
		$\langle W,W,W,W,XYZ \rangle_\ast =1/20000$\\\hline
		7 pt correlators
	 	\\ \hline
		$\langle Y,W,W,W,W,XYZ,XYZ \rangle_\ast = -3/500000$\\
		$\langle Z,Z,W,W,W,XYZ,XYZ \rangle_\ast = 0$\\
\end{tabular}&
\begin{tabular}{c}
4 pt correlators
	 	\\ \hline
		$\langle X,X,Y,XYZ \rangle = 1/5$\\
		$\langle X,X,XY,YZ \rangle = 1/5$\\
		$\langle X,X,XZ,XZ \rangle = 2/5$\\
		$\langle X,Y,XY,XZ \rangle= 1/5$\\
		$\langle X,Y,X^2,YZ \rangle_\ast = 2/5$\\
		$\langle X,Z,Z,YZ \rangle =1/5$\\
		$\langle X,Z,XY,XY \rangle =0$\\
		$\langle X,Z,X^2,XZ \rangle_\ast =2/5$\\
		$\langle X,W,W,XY \rangle_\ast =-1/100$\\
		$\langle Y,Y,Z,YZ \rangle =1/10$\\
		$\langle Y,Y,XY,XY \rangle =1/5$\\
		$\langle Y,Y,X^2,XZ \rangle_\ast =0$\\
		$\langle Y,Z,Z,XZ \rangle =1/5$\\
		$\langle Y,Z,X^2,XY \rangle_\ast =0$\\
		$\langle Y,W,W,X^2 \rangle_\ast =-1/100$\\
		$\langle Z,Z,Z,XY \rangle =1/10$\\
		$\langle Z,Z,W,W \rangle_\ast =-1/100$\\
		$\langle Z,Z,X^2,X^2 \rangle_\ast =4/5$\\\hline
		6 pt correlators
	 	\\ \hline
		$\langle X,X,W,W,XYZ,XYZ \rangle_\ast = -1/1250$\\
		$\langle X,Z,Z,Z,XYZ,XYZ \rangle = 0$\\
		$\langle Y,Y,Z,Z,XYZ,XYZ \rangle =2/125$\\
		$\langle Y,Z,W,W,YZ,XYZ \rangle_\ast =1/10000$\\
		$\langle Z,Z,W,W,XZ, XYZ \rangle_\ast =-3/5000$\\
		$\langle W,W,W,W,X^2,XYZ \rangle_\ast =-9/100000$\\
		$\langle W,W,W,W,YZ ,YZ\rangle_\ast =3/100000$\\
\end{tabular}\\
\hline\hline
\end{tabular}

\begin{tabular}{||c|c||}
\hline\hline
\multicolumn{2}{||c||}{$Q_{17}=x^3+xy^5+yz^2, \quad G^{max} = \langle J = (1/3,2/15,13/30)\rangle$} \\\hline\hline
\multicolumn{2}{||c||}{$k \leq 5 $, $X = e_{10J}, Y =e_{8J}$. $P = 3/10, \hat{c} = 6/5$. }\\
\multicolumn{2}{||c||}{Relations:  $5X^9 = 2Y^3 $ , $X^{10} = Y^4 = 0$, $\dim =21$. }\\
\hline
 	\begin{tabular}{c}
 	{3 pt correlators}
 	\\ \hline
	\begin{tabular}{cc}
$\langle \one, \one, X^9Y\rangle= -2$&
$\langle \one, X, X^8Y\rangle= -2$\\
$\langle \one, X^2, X^7Y\rangle= -2$&
$\langle \one, Y, X^9\rangle= -2$\\
$\langle \one, X^3, X^6Y\rangle= -2$&
$\langle \one, XY, X^8\rangle= -2$\\
$\langle \one, X^4, X^5Y\rangle= -2$&
$\langle \one, X^5, X^4Y\rangle= -2$\\
$\langle \one, X^2Y, X^7\rangle= -2$&
$\langle \one, X^3Y, X^6\rangle= -2$\\
$\langle \one, Y^2, Y^2\rangle= -5$&
$\langle X, X, X^7Y\rangle= -2$\\
$\langle X, X^2, X^6Y\rangle= -2$&
$\langle X, Y, X^8\rangle= -2$\\
$\langle X, X^3, X^5Y\rangle= -2$&
$\langle X, XY, X^7\rangle= -2$\\
$\langle X, X^4, X^4Y\rangle_\ast= -2$&
$\langle X, X^5, X^3Y\rangle= -2$\\
$\langle X, X^2Y, X^6\rangle= -2$&
$\langle X^2, X^2, X^5Y\rangle= -2$\\
$\langle X^2, Y, X^7\rangle= -2$&
$\langle X^2, X^3, X^4Y\rangle_\ast= -2$\\
$\langle X^2, XY, X^6\rangle= -2$&
$\langle X^2, X^4, X^3Y\rangle_\ast= -2$\\
$\langle X^2, X^2Y, X^5\rangle= -2$&
$\langle Y, Y, Y^2\rangle_\ast= -5$\\
$\langle Y, X^3, X^6\rangle= -2$&
$\langle Y, X^4, X^5\rangle= -2$\\
$\langle X^3, X^3 X^3Y\rangle_\ast= -2$&
$\langle X^3, XY X^5\rangle_\ast= -2$\\
$\langle X^3, X^4 X^2Y\rangle= -2$&
$\langle XY, X^4 X^4\rangle_\ast= -2$\\

\end{tabular}
		\end{tabular}&
\begin{tabular}{c}
4 pt correlators
	 	\\ \hline
		$\langle X,X,X^8,X^9Y \rangle = 8/15$\\
		$\langle X,X,X^9,X^8Y \rangle = 4/15$\\
		$\langle X,Y,Y^2,X^9Y \rangle_\ast = 2/3$\\
		$\langle X,Y,X^9,X^9 \rangle = -4/15$\\
		$\langle Y,Y,XY,X^9Y \rangle =-2/3$\\
		$\langle Y,Y,X^2Y,X^8Y \rangle =-2/3$\\
		$\langle Y,Y,X^3Y,X^7Y \rangle =-2/3$\\
		$\langle Y,Y,X^4Y,X^6Y \rangle =-2/3$\\
		$\langle Y,Y,X^5Y,X^5Y \rangle =-2/3$\\
		5 pt correlators
	 	\\ \hline
		None
\end{tabular}\\
\hline\hline
\end{tabular}

\begin{tabular}{||c|c||}
\hline\hline
\multicolumn{2}{||c||}{$Q_{17}^T=x^3y+y^5z+z^2, \quad G^{max} = \langle \gamma = (1/30,9/10,1/2)\rangle$} \\\hline\hline
\multicolumn{2}{||c||}{$k \leq 5 $, $X = e_{7\gamma}, Y =e_{19\gamma}, Z = y^4e_{10\gamma}$. $P = 1/3, \hat{c} = 6/5$. }\\
\multicolumn{2}{||c||}{Relations:  $X^5 = -3Y^2, aX^4Y = Z^2, X^{10} = Y^3 = Z^3 = 0 $, $a \neq 0$,  $\dim = 17$.  }\\
\hline
 	\begin{tabular}{c}
 	{3 pt correlators}
 	\\ \hline
	\begin{tabular}{cc}
$\langle \one, \one, X^4Y^2\rangle= 1$&
$\langle \one, X, X^3Y^2\rangle= 1$\\ 
$\langle \one, X^2, X^2Y^2\rangle= 1$&
$\langle \one, Y, X^4Y\rangle= 1$\\ 
$\langle \one, X^3, XY^2\rangle= 1$&
$\langle \one, XY, X^3Y\rangle= 1$\\
$\langle \one, X^4, Y^2\rangle= 1$&
$\langle \one, X^2Y, X^2Y\rangle= 1$\\ 
$\langle \one, Z, YZ\rangle= a$&
$\langle X, X, X^2Y^2\rangle= 1$\\ 
$\langle X, X^2, XY^2\rangle= 1$&
$\langle X, Y, X^3Y\rangle= 1$\\ 
$\langle X, X^3, Y^2\rangle= 1$&
$\langle X, XY, X^2Y\rangle= 1$\\
$\langle X, X^4, X^4\rangle_\ast= -3$&
$\langle X^2, X^2, Y^2\rangle= 1$\\ 
$\langle X^2, Y, X^2Y\rangle= 1$&
$\langle X^2, X^3, X^4\rangle_\ast= -3$\\ 
$\langle X^2, XY, XY\rangle= 1$&
$\langle Y, Y, X^4\rangle= 1$\\ 
$\langle Y, X^3, XY\rangle= 1$&
$\langle Y, Z, Z\rangle_\ast= a$\\ 
$\langle X^3, X^3, X^3\rangle_\ast= -3$\\
\end{tabular}\\\\\\\\\hline
		5 pt correlators
	 	\\ \hline
		$\langle X,Y,Y,X^4Y^2,X^4Y^2 \rangle = -1/150$\\
		$\langle Y,Z,Z,XY^2,X^4Y^2 \rangle = -a/150$\\
		$\langle Y, Z,Z,X^2Y^2,X^3Y^2 \rangle = -a/150$\\
\end{tabular}&
\begin{tabular}{c}
4 pt correlators
	 	\\ \hline
		$\langle X,X,X^3Y,X^4Y^2 \rangle = 1/10$\\
		$\langle X,X,X^4Y,X^3Y^2 \rangle = 1/10$\\
		$\langle X,Y,Y^2,X^3Y^2 \rangle = -1/15$\\
		$\langle X,Y,XY^2,X^2Y^2 \rangle = -1/15$\\
		$\langle X,Y,X^4Y,X^4Y \rangle = 0$\\
		$\langle X,Y,X^4,X^4Y^2 \rangle_\ast =1/10$\\
		$\langle X,Z,Z,X^4Y^2 \rangle_\ast =-a/10 $\\
		$\langle X,Z,YZ,X^4Y \rangle_\ast = -a/10$\\
		$\langle X^2,Z,Z,X^3Y^2 \rangle_\ast = -a/10$\\
		$\langle Y,Y,Y,X^4Y^2 \rangle =3/10$\\
		$\langle Y,Y,XY,X^3Y^2 \rangle =7/30$\\
		$\langle Y,Y,X^2Y,X^2Y^2 \rangle =1/6$\\
		$\langle Y,Y,Y^2,X^4Y \rangle =1/30$\\
		$\langle Y,Y,X^3Y,XY^2 \rangle =1/10$\\
		$\langle Y,Y,YZ,YZ \rangle_\ast =a/3$\\
		$\langle Y,Z,Y^2,YZ \rangle_\ast =a/30$\\
		$\langle X^3,Y,Y^2,X^4Y \rangle_\ast =-a/10$\\
		$\langle Z,Z,XY,X^4Y \rangle_\ast =-a/10$\\
		$\langle Z,Z,X^4,XY^2 \rangle_\ast =-a/10$\\
		$\langle Z,Z,X^2Y,X^3Y \rangle_\ast =-a/10$\\
		$\langle Z,Z,Y^2,Y^2 \rangle_\ast =a/30$\\
\end{tabular}\\
\hline\hline
\end{tabular}

\begin{tabular}{||c|c||}
\hline\hline
\multicolumn{2}{||c||}{$S_{16}=x^2y+xz^4+y^2z, \quad G^{max} =  \langle J = (5/17,7/17,3/17)\rangle$} \\\hline\hline
\multicolumn{2}{||c||}{$k \leq 5$, $X = e_{8J}, Y =e_{7J}, Z = e_{6J}$. $P = 7/17, \hat{c} = 21/17$. }\\
\multicolumn{2}{||c||}{Relations:  $-2XZ = Y^2, X^4 = -2YZ, -4X^3Y = Z^2, X^4 = Y^3 = Z^3 = 0$, $\dim =16$.  }\\
\hline
 	\begin{tabular}{l}
 	{3 pt correlators}
 	\\ \hline
$\langle \one, \one, X^3YZ\rangle= 1$\\
$\langle \one, X, X^2YZ\rangle= 1$\\ 
$\langle \one, Y, X^3Z\rangle= 1$\\
$\langle \one, X^2, XYZ\rangle= 1$\\ 
$\langle \one, Z, X^3Y\rangle= 1$\\
$\langle \one, XY, X^2Z\rangle= 1$\\ 
$\langle \one, X^3, YZ\rangle= 1$\\
$\langle \one, XZ, X^2Y\rangle= 1$\\ 
$\langle X, X, XYZ\rangle= 1$\\
$\langle X, Y, X^2Z\rangle= 1$\\ 
$\langle X, X^2, YZ\rangle= 1$\\
$\langle X, Z X^2Y\rangle= 1$\\ 
$\langle X, XY, XZ\rangle= 1$\\
$\langle X, X^3, Y^3\rangle_\ast= -2$\\ 
$\langle Y, X^2, XZ\rangle= 1$\\
$\langle Y, Z, X^3\rangle= 1$\\ 
$\langle Y, Y, X^2Y\rangle_\ast= -2$\\ 
$\langle Y, XY, XY\rangle_\ast= -2$\\ 
$\langle X^2, Z, XY\rangle= 1$\\
$\langle X^2, X^2, X^3\rangle_\ast= -2$\\ 
$\langle Z, Z, Z\rangle_\ast= -4$\\  
\end{tabular}&
\begin{tabular}{c}
4 pt correlators
	 	\\ \hline
		\begin{tabular}{ll}
		\begin{tabular}{l}
		$\langle X,X,X^2Y,X^3YZ \rangle = 3/17$\\
		$\langle X,X,X^3Y,X^2YZ \rangle = 1/17$\\
		$\langle X,X,X^3Z,X^3Z \rangle_\ast = -2/17$\\
		$\langle X,Y,YZ,X^2YZ \rangle = -2/17$\\
		$\langle X,Y,X^3Y,X^3Z \rangle = 1/17$\\
		$\langle X,Y,XYZ,XYZ \rangle = -2/17$\\
		$\langle X,Y,X^3,X^3YZ \rangle_\ast = 3/17$\\
		$\langle X,Z,XZ,X^2YZ \rangle = 1/17$\\
		$\langle X,Z,YZ,X^3Z \rangle = 1/17$\\
		$\langle X,Z,X^2Z,XYZ \rangle = 1/17$\\
		$\langle X,Z,X^3Y,X^3Y \rangle = -2/17$\\
		$\langle X,Z,Z,X^3YZ \rangle_\ast = 5/17$\\
		$\langle Y,Y,Z,X^3YZ \rangle_\ast =7/17$\\
		$\langle Y,Y,XZ,X^2YZ \rangle_\ast =5/17$\\
		\end{tabular}&
		\begin{tabular}{l}
		$\langle Y,Y,YZ,X^3Z \rangle_\ast =1/17$\\
		$\langle Y,Y,X^2Z,XYZ \rangle_\ast =3/17$\\
		$\langle Y,Y,X^3Y,X^3Y \rangle_\ast =-2/17$\\
		$\langle Y,Z,XZ,X^4Y \rangle =-4/17$\\
		$\langle Y,Z,YZ,X^4Y \rangle =1/17$\\
		$\langle Y,Z,X^2Z,X^4Y \rangle =-4/17$\\
		$\langle Y,Z,XY,X^4Y \rangle_\ast =5/17$\\
		$\langle Y,Z,X^2Y,X^4Y \rangle_\ast =3/17$\\
		$\langle X^2,Z,Z,X^2YZ \rangle_\ast =6/17$\\
		$\langle Z,Z,XY,X^3Z \rangle_\ast =1/17$\\
		$\langle Z,Z,X^3,XYZ \rangle_\ast =7/17$\\
		$\langle Z,Z,XZ,X^3Y \rangle_\ast =3/17$\\
		$\langle Z,Z,X^2Y,X^2Z \rangle_\ast =2/17$\\
		$\langle Z,Z,YZ,YZ \rangle_\ast =-4/17$\\
		\end{tabular}
		\end{tabular}
		\\\hline
		5 pt correlators
	 	\\ \hline
		$\langle X,X,Z,X^3YZ,X^3YZ \rangle_\ast =-4/289$\\
		$\langle X,Y,Y,X^3YZ,X^3YZ \rangle_\ast =8/289$\\
		$\langle Y,Z,Z,XYZ,X^3YZ \rangle_\ast =-4/289$\\
		$\langle Y,Z,Z,X^2YZ,X^2YZ \rangle_\ast =-8/289$\\
		$\langle Z,Z,Z,X^2Z,X^3YZ \rangle_\ast =20/289$\\
		$\langle Z,Z,Z,X^3Z,X^2YZ \rangle_\ast =12/289$\\
\end{tabular}\\
\hline\hline
\end{tabular}

\begin{tabular}{||c|c||}
\hline\hline
\multicolumn{2}{||c||}{$S_{17}=x^2y+y^2z+z^6, \quad G^{max} = \langle J = (7/24,5/12,1/6)\rangle$} \\\hline\hline
\multicolumn{2}{||c||}{$k \leq 5 $, $X = e_{8J}, Y =e_{7J}$. $P = 1/4, \hat{c} = 5/4$. }\\
\multicolumn{2}{||c||}{Relations:  $X^6 =Y^3, X^{11} = Y^4 = 0 $, $\dim = 21$.  }\\
\hline
 	\begin{tabular}{cc}
 	\multicolumn{2}{c}{3 pt correlators}
 	\\ \hline
$\langle \one, \one, X^{10}\rangle= -2$&
$\langle \one, X, X^9\rangle= -2$\\ 
$\langle \one, Y, X^4Y^2\rangle= -2$&
$\langle \one, X^2, X^8\rangle= -2$\\
$\langle \one, XY, X^3Y^2\rangle= -2$& 
$\langle \one, X^3, X^7\rangle= -2$\\ 
$\langle \one, X^2Y, X^2Y^2\rangle= -2$&
$\langle \one, X^4, X^6\rangle= -2$\\
$\langle \one, Y^2, X^4Y\rangle= -2$& 
$\langle \one, X^3Y, XY^2\rangle= -2$\\
$\langle \one, X^5, X^5\rangle= -2$&
$\langle X, X, X^8\rangle= -2$\\
$\langle X, Y, X^3Y\rangle= -2$& 
$\langle X, X^2, X^7\rangle= -2$\\
$\langle X, XY, X^2Y^2\rangle= -2$&
$\langle X, X^3, X^6\rangle= -2$\\ 
$\langle X, X^2Y, XY^2\rangle= -2$& 
$\langle X, Y^2, X^3Y\rangle= -2$\\
$\langle X, X^4, X^5 \rangle= -2$& 
$\langle Y, Y, X^4Y\rangle_\ast= -2$\\ 
$\langle Y, X^2, X^2Y^2\rangle= -2$&
$\langle Y, XY, X^3Y\rangle_\ast= -2$\\ 
$\langle Y, X^3, XY^2\rangle= -2$&
$\langle Y, X^2Y, X^2Y\rangle_\ast= -2$\\ 
$\langle Y, X^4, Y^2\rangle= -2$&
$\langle X^2, X^2, X^6\rangle= -2$\\ 
$\langle X^2, XY, XY^2\rangle= -2$&
$\langle X^2, X^3, X^5\rangle= -2$\\ 
$\langle X^2, X^2Y, Y^2\rangle= -2$& 
$\langle X^2, X^4, X^4\rangle_\ast= -2$\\ 
$\langle XY, X^3, Y^2\rangle= -2$&
$\langle XY, XY, X^2Y\rangle_\ast= -2$\\
$\langle X^3, X^3, X^4\rangle_\ast= -2$&
\end{tabular}&
\begin{tabular}{c}
4 pt correlators
	 	\\ \hline
		$\langle X,X,X^4Y,X^{10} \rangle = -1/3$\\
		$\langle X,X,X^3Y^2,X^3Y^2 \rangle_\ast = 1/(3k)$\\
		$\langle X,Y,X^5,X^{10} \rangle_\ast = -1/3$\\
		$\langle X,Y,X^6,X^9 \rangle = -1/3$\\
		$\langle X,Y,X^7,X^8 \rangle = -1/3$\\
		$\langle Y,Y,Y^2,X^{10} \rangle_\ast =5/3$\\
		$\langle Y,Y,XY^2,X^9 \rangle_\ast =-1/3$\\
		$\langle Y,Y,X^6,X^3Y^2 \rangle_\ast =1/3$\\
		$\langle Y,Y,X^2Y^2,X^8 \rangle_\ast =1$\\
		$\langle Y,Y,Y^2,X^4Y \rangle_\ast =1/18$\\
		$\langle Y,Y,X^7,X^3Y^2 \rangle_\ast =-1/3$\\\\\\\hline
		5 pt correlators
	 	\\ \hline
		None\\\\
\end{tabular}\\
\hline\hline
\end{tabular}

}

}

\subsection{ FJRW Theories which we still cannot compute}\label{sec_nogo}
Unfortunately, there are some A-models which we are still unable to compute.  In these examples, there are not enough relations between concave correlators (whose values we can compute) and other correlators (whose values we can only find using the Reconstruction Lemma).
{\small 

\begin{center}
\begin{tabular}{| c | l | c |}
\hline
$W$ & $G$ & dimension  \\\hline
$P_8 = x^3+y^3+z^3$ &  $\langle (1/3,0,0),(0,1/3,1/3)\rangle$ &8 \\
$U_{12} = x^3+y^3+z^4$ &  $\langle J \rangle$ &12 \\
$S_{1,0}^T = x^2+xz^2 + y^5z$ &  $\langle J \rangle_{max}$ &14 \\
$Z_{17} = x^3y+y^8$ &  $\langle J \rangle_{max}$ &22 \\
$Z_{18} = x^3y + xy^6$ &  $\langle J \rangle_{max}$ &18 \\
$W_{17}^T = x^4y + y^5$ &  $\langle J \rangle $ &8 \\
$Q_{17}^T=x^3y+y^5z+z^2$ &  $\langle J \rangle $ &7 \\
$S_{17}^T=x^2+y^6z+z^2x$ &  $ G^{max}$ &17 \\
$S_{17}^T=x^2+y^6z+z^2x$ &  $\langle J \rangle $ &7 \\
\hline
\end{tabular} 
\end{center}

}

\begin{remark}\label{rem:code}
The code which we used to make our computations is available by email request from the author.  We made the computations in SAGE \cite{sage}.  The SAGE computations depend on code written by Drew Johnson \cite{Jo}   for computing intersection numbers of classes on $\overline{M}_{g,n}$.  
\end{remark}

\bibliography{references}{}
\bibliographystyle{unsrt}

\end{document}